\documentclass[12pt]{amsart}

\usepackage{amsmath,amssymb,amscd,amsthm,graphicx,enumerate}

\usepackage{hyperref}
\hypersetup{breaklinks=true}

\numberwithin{equation}{section}
\theoremstyle{plain}

\makeatletter
\newtheorem*{rep@theorem}{\rep@title}
\newcommand{\newreptheorem}[2]{%
\newenvironment{rep#1}[1]{%
 \def\rep@title{#2 \ref{##1}}%
 \begin{rep@theorem}}%
 {\end{rep@theorem}}}
\makeatother

\newtheorem{theorem}[equation]{Theorem}
\newreptheorem{theorem}{Theorem}

\newtheorem{proposition}[equation]{Proposition}
\newtheorem{lemma}[equation]{Lemma}
\newtheorem{corollary}[equation]{Corollary}

\newtheorem{claim}[equation]{Claim}

\newtheorem*{theorem18p}{Theorem 1.8'}

\theoremstyle{remark}
\newtheorem{remark}[equation]{Remark}

\theoremstyle{definition}
\newtheorem{definition}[equation]{Definition}

\newtheorem*{question*}{Question}

\newcommand{\K}{{\mathcal K}}
\renewcommand{\L}{{\mathcal L}}
\newcommand{\M}{{\mathcal M}}

\newcommand{\Q}{\mathbb Q}
\newcommand{\R}{\mathbb R}

\renewcommand{\S}{{\mathcal S}}

\newcommand{\cone}{\operatorname{Cone}}

\newcommand{\Div}{\operatorname{div}}
\newcommand{\Ext}{\operatorname{Ext}}

\newcommand{\Int}{\operatorname{Int}}

\newcommand{\al}{\alpha}
\newcommand{\be}{\beta}
\newcommand{\D}{\partial}
\newcommand{\de}{\delta}

\newcommand{\Lap}{\Delta}

\newcommand{\eps}{\varepsilon}
\newcommand{\ga}{\gamma}

\newcommand{\la}{\lambda}
\newcommand{\La}{\Lambda}

\newcommand{\ol}{\overline}

\newcommand{\ra}{\rightarrow}

\providecommand{\abs}[1]{\lvert #1\rvert}
\providecommand{\norm}[1]{\lvert\lvert #1\rvert\rvert}

\def\XXint#1#2#3{{\setbox0=\hbox{$#1{#2#3}{\int}$}
     \vcenter{\hbox{$#2#3$}}\kern-.5\wd0}}

\begin{document}

\title[Mean curvature flow of mean convex hypersurfaces]{Mean curvature flow of mean convex hypersurfaces}
\author{Robert Haslhofer}
\author{Bruce Kleiner}

\date{\today}
\maketitle


\begin{abstract}
In the last 15 years, White and Huisken-Sinestrari 
developed a far-reaching structure theory for the mean curvature flow of mean convex hypersurfaces.
Their papers \cite{white_size,white_nature,white_subsequent,huisken-sinestrari1,huisken-sinestrari2,huisken-sinestrari3}  provide a package of estimates and structural results that yield a precise description of singularities and
 of high curvature regions in a mean convex flow.

In the present paper, we give a new treatment of the theory of mean convex 
(and $k$-convex) flows. This includes: (1) an estimate for derivatives of curvatures, 
(2) a convexity 
estimate, (3) 
a cylindrical estimate, (4) a global convergence theorem, (5) 
a structure theorem for ancient solutions, and (6) a partial regularity theorem.

Our new proofs are both more elementary and substantially shorter than the original arguments. Our estimates are local and universal.
A key ingredient in our new approach is the new noncollapsing result of Andrews \cite{andrews1}. Some parts are also inspired by the work of Perelman \cite{perelman_entropy,perelman_surgery}.
In a forthcoming paper \cite{haslhofer-kleiner_surgery},
we will give a new construction of mean curvature flow with surgery based on the methods established in the present paper.
\end{abstract}

\tableofcontents

\section{Introduction}

\subsection{Background and history}
The mean curvature flow evolves hypersurfaces in time; the velocity is given by the mean curvature vector.  Mean curvature flow has been 
extensively studied since the pioneering work of Brakke \cite{brakke},  
as it is the most natural parabolic evolution equation for a moving
submanifold.
While the theory was progressing in many fruitful directions, 
there was one persistent central theme:  
the investigation of the structure of singularities, and the development of
 related techniques.  In the last 15 years, this
culminated in the spectacular
work of White \cite{white_size,white_nature,white_subsequent} and Huisken-Sinestrari 
\cite{huisken-sinestrari1,huisken-sinestrari2,huisken-sinestrari3}
on mean curvature flow in the case of mean convex 
hypersurfaces, i.e. hypersurfaces with nonnegative mean curvature.
Their papers give a far-reaching structure theory, providing a package of estimates
that yield a qualitative picture of singularities and a global description of the
large curvature part in a mean convex flow.

\noindent\emph{White's theory in a nutshell.} White uses weak versions of 
mean curvature flow 
-- level set flow \cite{evans-spruck,CGG} and Brakke flow \cite{brakke,Ilmanen} -- 
and the language
of geometric measure theory. Given any mean convex hypersurface $M_0\subset \R^N$ 
(smooth, compact, embedded) there is a unique weak flow (that can be described either as 
level set flow or as Brakke flow) starting at $M_0$.
White's two main theorems describe the size and the nature of the singular set: 
(1) The singular set has parabolic Hausdorff dimension at most $N-2$ 
\cite[Thm. 1]{white_size}; (2) Every limit flow is smooth and convex (a limit flow is a 
weak limit of any sequence of parabolic rescalings where the scaling factors tend to 
infinity) -- in particular, all tangent flows are shrinking spheres, cylinders or planes of
multiplicity one \cite[Thm. 1]{white_nature}, 
\cite[Thm. 3]{white_subsequent}.\footnote{In fact, White proves most of his results in a 
setting that also allows general ambient manifolds and
nonsmooth initial conditions: (1) still holds and (2) also holds
provided $N\leq 7$.} 

\noindent\emph{Huisken-Sinestrari theory in a nutshell.} Huisken-Sinestrari use the language of
smooth differential geometry and partial differential equations.
Their three main estimates are: (1) the convexity estimate \cite[Thm. 3.1]{huisken-sinestrari1}, \cite[Thm. 1.1]{huisken-sinestrari2}, \cite[Thm. 1.4]{huisken-sinestrari3}, which says that regions of high curvature have almost positive definite second fundamental form; (2) the cylindrical estimate \cite[Thm. 1.5]{huisken-sinestrari3} which says that in a $2$-convex flow, i.e. a flow where the sum of the smallest two principal curvatures is nonnegative, regions of high curvature are either uniformly convex or close to a cylinder (assuming $N\geq 4$);
(3) the gradient estimate \cite[Thm. 1.6]{huisken-sinestrari3} which gives derivative bounds depending only at the curvature at a single point.
Based on these estimates, Huisken-Sinestrari succeeded in constructing a mean curvature flow with surgery for $2$-convex hypersurfaces of dimension at least three \cite{huisken-sinestrari3}.
This has elements in common with -- but also some interesting differences with --  
Perelman's Ricci flow with surgery \cite{perelman_entropy,perelman_surgery}.

\noindent\emph{Multiplicity one and noncollapsing.} A crucial step in White's work is ruling 
out multiplicity two hyperplanes as potential blowup limits.
To do this, White uses a sophisticated line of clever arguments (the expanding hole theorem, 
the sheeting theorem, a Bernstein theorem, etc.), see \cite{white_size}.
A similar issue (ruling out the possibility of collapsing with bounded curvature) arises in 
the Ricci flow, where it was handled by Perelman's $\mathcal{W}$-functional and $L$-function
 \cite{perelman_entropy}.  Since Perelman's work, results of this kind --
for either mean curvature flow or Ricci flow -- have been called 
{\em noncollapsing} theorems. In a beautiful recent paper \cite{andrews1}, Andrews has given a 
direct proof of noncollapsing for mean convex flows, using only the maximum principle.

\noindent\emph{Convexity estimate.} Another crucial step in the theory is the convexity 
estimate. It says that points with large mean curvature have almost positive definite second
fundamental form. This is similar to the Hamilton-Ivey pinching estimate for three-dimensional
Ricci flow \cite{Hamilton_survey}, but its proof is 
substantially more involved. The 
Huisken-Sinestrari proof involves  a sophisticated iteration scheme with 
$L^p$-estimates, the Michael-Simons-Sobolov inequality, and 
recursion formulas for the symmetric polynomials in the principal curvatures 
\cite{huisken-sinestrari1,huisken-sinestrari2,huisken-sinestrari3}. 
White's proof that blowup limits have positive semidefinite second fundamental form is based instead on the rigidity case of the maximum principle combined with much of his
geometric-measure-theoretic regularity and structure theory \cite{white_size,white_nature}.

\noindent\emph{Related work.} Using rather different techniques, some aspects of convex solutions and the singularity structure in mean convex flows have been studied by Wang \cite{wang_convex} and Sheng-Wang \cite{sheng_wang}. 
The structure of the singular set in $k$-convex flows has been studied by Head \cite{Head_thesis}, Ecker \cite{Ecker_kconvex}, and Cheeger-Haslhofer-Naber \cite{CHN}.

\subsection{Overview}
In the present paper, we give a new treatment of the theory of mean convex (and $k$-convex) mean curvature flow.

Starting with a quick panorama, the key parts of the theory are the curvature estimate
(Theorem \ref{thm-intro_local_curvature_bounds}), the convexity estimate
(Theorem \ref{thm-intro_convexity_estimate}),
the cylindrical estimate (Theorem \ref{thm-intro_cylindrical}), 
the global convergence theorem (Theorem \ref{thm-intro_h_gives_global_convergence}), 
the structure theorem for ancient solutions (Theorem \ref{thm-intro_smooth_convex_til_extinct}), 
and the partial regularity theorem (Theorem \ref{thm_intro_partial_regularity}).

As in \cite{white_size,white_nature,white_subsequent}, our theorems can be viewed as results about the level set flow. On the other hand, by making use of a viscosity notion of mean curvature 
(Definition \ref{def_viscosity_mean_curvature}) and approximation we are able to formulate all our estimates and proofs
as if we were dealing with smooth mean curvature flow, as in 
\cite{huisken-sinestrari1,huisken-sinestrari2,huisken-sinestrari3}.

A key ingredient in our new approach is the beautiful 
noncollapsing result of Andrews \cite{andrews1}.  We found it  natural
to build his result into our framework by working with 
a class of flows satisfying the conclusion of his theorem; we call 
them $\al$-Andrews flows. In particular, this class contains all mean convex level set flows with smooth initial data.

Our new proofs are both more elementary and 
substantially shorter than the original arguments. 
The reduction in length is illustrated most dramatically by our new proofs of the curvature 
and the convexity estimate.  Together they take only three pages 
(see Section \ref{sec-halfspace_and_consequences}), as opposed to a couple of sophisticated
papers in the original arguments (see background section). For the benefit of readers with little or no prior experience with the mean curvature flow, we have made the exposition as self-contained as possible, and also included several appendices explaining some background material.

Apart from streamlining the proofs, our results are local and depend only on the value of the Andrews constant. We will exploit this local and universal character in our forthcoming paper \cite{haslhofer-kleiner_surgery}, to 
give a new and general construction of mean curvature flow with surgery; in particular, our new construction also works in the case of mean convex surfaces in $\mathbb{R}^3$.


\subsection{Notation and terminology}
We will now recall some standard notions needed for the present paper. For comprehensive 
introductions to the mean curvature flow with a focus on the formation of singularities we 
refer to the books by Ecker \cite{Ecker_book} and Mantegazza \cite{Mantegazza_book}.
We also warmly recommend White's ICM-survey \cite{White_ICM}.

The ambient space $\R^N$ is always assumed to have dimension $N\geq 2$, and we suppress the dependence of constants on $N$. 

A smooth family $\{M_t\subset\R^N\}_{t\in I}$
of closed embedded hypersurfaces, where $I\subset \R$
is an interval, {\em moves by mean curvature flow}
if $M_t=x_t( M)$ for some smooth family of embeddings 
$\{x_t: M\ra \R^N\}_{t\in I}$
satisfying the mean curvature flow equation
$$\frac{\D x_t}{\D t}= \vec{H}_t\,,$$
where $\vec{H}_t: M\ra \R^N$ denotes the mean curvature of $x_t$. Instead of the family $\{M_t\}$ itself, we will typically think
in terms of the evolving family $\{K_t\}$ of the compact domains bounded by the $M_t$'s. More generally, we sometimes also consider families of possibly noncompact closed domains $\{K_t\subseteq U\}_{t\in I}$ ($I\subseteq \R$) in an open set $U\subseteq \R^N$, whose boundaries move by mean curvature flow.

For the mean curvature flow, time scales like distance squared. {\em  Spacetime} $\R^{N,1}$ is defined to be 
$\R^N\times\R$ equipped
with the parabolic metric $d((x_1,t_1),(x_2,t_2))=\max({|x_1-x_2|,|t_1-t_2|^{\frac12}})$.
{\em Parabolic rescaling by $\la\in(0,\infty)$ at $(x_0,t_0)\in\R^{N,1}$} is described by the mapping
$$(x,t)\mapsto (\la (x-x_0),\la^2(t-t_0)).$$
The {\em parabolic ball} with radius $r>0$ and center
$X=(x,t)\in \R^{N,1}$ is the product 
$$
P(x,t,r)=B(x,r)\times (t-r^2,t]\subset \R^{N,1}.
$$
Given a family of subsets $\{K_t\subseteq \R^N\}_{t\in I}$ the {\em spacetime track} is the set
$$
\K=\cup_{t\in I}\;K_t\times \{t\}\subseteq \R^{N,1}\,.
$$
Likewise, given a subset $\K\subseteq \R^{N,1}$, the {\em time $t$ slice} of $\K$
is $$K_t=\{x\in \R^N\mid (x,t)\in \K\}\,.$$ 
We tend to conflate subsets of spacetime with
the corresponding family of time slices.

Given a compact subset $K_0\subset\R^N$, there
is a canonical
family $\{K_t\}_{t\geq 0}$ of closed sets starting at $K_0$, 
the {\em level set flow} of $K_0$. The level set flow has been introduced in 
Evans-Spruck \cite{evans-spruck} and Chen-Giga-Goto \cite{CGG}, and later a completely elementary characterization has been given by 
Ilmanen \cite{Ilmanen}: 
$\{K_t\}_{t\geq 0}$ is the \emph{maximal set flow} starting at $K_0$. We recall that a \emph{set flow} $\{C_t\}$ simply is a family of closed sets satisfying the \emph{avoidance principle}
\begin{equation*}
 C_{t_0}\cap Q_{t_0}=\emptyset\qquad\Rightarrow \qquad C_{t}\cap Q_{t}=\emptyset
\quad \textrm{for all}\,\, t\in[t_0,t_1],
\end{equation*}
whenever $\{Q_t\}_{t\in [t_0,t_1]}$ is a smooth compact mean curvature flow.
When  $K_0\subset\R^N$ is a compact
domain with smooth boundary, the level set 
flow of $K_0$ coincides with
smooth mean curvature flow of $K_0$ for as long as the latter is defined.  
The level set flow $\{K_t\}$ is called \emph{mean convex} if $K_{t_2}\subseteq K_{t_1}$ whenever $t_2\geq t_1$.\footnote{This differs slighly from the definition in \cite{white_size}: we do not require strict inclusion $K_{t_2}\subsetneq K_{t_1}$ for $t_2>t_1$.} Any level set flow starting at a smooth compact domain with nonnegative mean curvature is mean convex.

\subsection{Andrews condition and $\al$-Andrews flows}\label{sec_main_thm_etc}
Before turning to the regularity theory, we begin with the noncollapsing result
of Andrews and a generalization to the nonsmooth setting. This is a very important ingredient in our
treatment of mean convex flows.

\begin{figure}[h!] 
\begin{center}  
\includegraphics[scale=1]{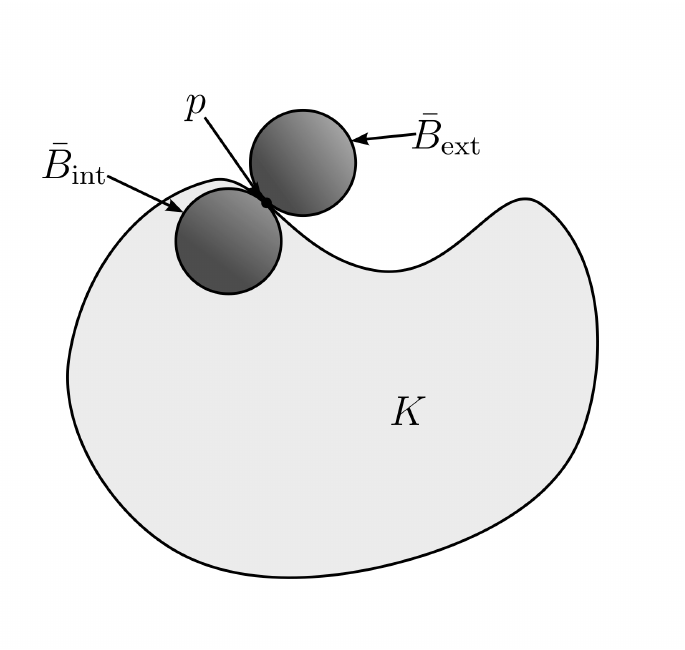} 
\caption{\label{fig-1}}
\end{center}
\end{figure}

\bigskip
\begin{definition}[Andrews condition]
\label{def-al_andrews}
If $K\subseteq\R^N$ is a smooth, closed,  mean convex domain 
(possibly disconnected) and $\al>0$, 
then $K$ satisfies the {\em $\al$-Andrews condition} if
for every $p\in \partial K$ there are closed balls $\bar{B}_{\Int}\subseteq K$ and 
$\bar{B}_{\Ext}\subseteq\R^N\setminus \Int({K})$ of radius at least $\frac{\al}{H(p)}$ that are tangent to $\D K$ at $p$ from the interior and exterior of 
$K$, respectively, as in Figure \ref{fig-1} (in the limiting case $H(p)=0$ this means that $K$ is a halfspace). 
A smooth mean curvature flow $\{K_t\subseteq \R^N\}_{t\in I}$ is 
{\em $\al$-Andrews} if every time slice satisfies the $\al$-Andrews condition.\footnote{More generally, a smooth $\alpha$-Andrews flow $\{K_t\subseteq U\}_{t\in I}$ in an open set $U\subseteq \R^N$
is a smooth mean convex flow in $U$,
such that for every $p\in\D K_t$ the balls tangent at $p$ with radius $\frac{\al}{H(p)}$  satisfy $\bar{B}_{\Int}\cap U\subseteq K_t$ respectively $\bar{B}_{\Ext}\cap U\subseteq U\setminus\Int(K_t)$.}
\end{definition}

Note that every compact, smooth, strictly mean convex domain is $\al$-Andrews for
some $\al>0$.  
The main theorem of Andrews \cite{andrews1} is:

\begin{theorem}
If  the initial condition $K_0$ of a smooth compact
mean curvature flow $\{K_t\}$ is $\al$-Andrews, then so is the whole flow $\{K_t\}_{t\geq 0}$.
\end{theorem}
The Andrews condition strongly restricts the kind of singularities that can form
under the flow; in particular,
it immediately excludes multiplicity two planes and grim reapers as 
potential blowup limits. 

In order to fully exploit the Andrews condition, we first show that it persists after the formation of singularities.  To this end, we consider the following notion of mean  curvature for nonsmooth sets:

\begin{definition}[cf. \protect{\cite[Def. 7]{kleiner_isoperimetric_comparison}}]
\label{def_viscosity_mean_curvature}
Let $K\subseteq \R^N$ be a closed set.   If $p\in \D K$, then the {\em viscosity mean curvature of $K$ at $p$} is
$$
H(p)=\inf\{H_{\D X}(p)\mid X\subseteq K\;\text{is a compact smooth domain,}
\;  p\in \D X\},\\
$$
where $H_{\D X}(p)$ denotes the mean curvature of $\D X$ at $p$ with respect to the inward pointing normal.
By the usual convention, the infimum of the empty set is $\infty$.
\end{definition}

Using viscosity mean curvature, we may then define the $\al$-Andrews condition 
for closed sets and level set flows:
\begin{definition}
A closed set $K$ satisfies the  {\em viscosity $\al$-Andrews
condition} if $H(p)\in [0,\infty]$ at every boundary point, and there are interior and exterior balls $\bar B_{\Int}$,
$\bar B_{\Ext}$ passing through $p$,  as in Definition \ref{def-al_andrews}.
A level set flow $\{K_t\}$ is {\em $\al$-Andrews} if every time slice satisfies
the viscosity $\al$-Andrews condition.
\end{definition}

Using these definitions and elliptic regularization we can extend Andrews'
theorem to the nonsmooth setting:
\begin{theorem}
If $\{K_t\}$ is a compact level set flow with smooth $\al$-Andrews initial condition, then
$\{K_t\}$ is an $\al$-Andrews level set flow. 
\end{theorem}

At first sight, such $\al$-Andrews level set flows would
provide a natural setting for the treatment of mean convex mean curvature
flow.  However, since blow-up arguments involve passage to weak 
limits, we decided to work instead with the following larger class of flows.

\begin{definition}\label{def_andrews_flows}
 The class of {\em  $\al$-Andrews flows} is the smallest class of set 
flows which contains all compact $\al$-Andrews level set 
flows with smooth initial condition and all smooth $\al$-Andrews flows $\{K_t\subseteq U\}_{t\in I}$ in any open set $U\subseteq \R^N$, and which is closed under the operations of
restriction to open sets, restriction to smaller time
intervals, parabolic rescaling, and passage to Hausdorff limits.
\end{definition}

General $\al$-Andrews flows share some basic properties with $\al$-Andrews level set flows; in particular, two-sided avoidance, monotonicity, the viscosity Andrews condition, and $\eps$-regularity, see Corollary \ref{properties_of_alpha_andrews}.

In the present paper, we follow the philosophy of approximation by smooth flows, an idea
that has already been emphasized by Evans-Spruck \cite[Section 7]{evans-spruck} and White 
\cite{white_size,white_regularity,white_subsequent}.
Namely, in Theorem \ref{thm_ell_reg} we prove that any $\al$-Andrews flow $\{K_t\}$ (respectively its stabilized version $\{K_t\times \R\}$) can be approximated by smooth flows satisfying the Andrews condition.
This has the pleasant implication that we can implement almost the entire theory of mean convex flows
-- apart from the approximation argument 
itself -- in the setting of smooth flows.
We hope this makes the smoothly inclined readers feel comfortable; for a more axiomatic treatment, see Remark \ref{axiom_remark}.

\begin{remark}
In particular, our results apply to mean curvature flows in a spacetime region 
$\mathcal{U} = U \times [t_1,t_2]$, provided one can control the Andrews quantity 
along the parabolic boundary
$\D\mathcal{U}=(U \times \{t_1\})  \cup  (\partial U \times [t_1,t_2])$, 
and $H>0$ on $\D\mathcal{U}$. This could be useful for the study of generic mean curvature 
flow, introduced by Colding-Minicozzi \cite{CM_generic}.
\end{remark}

\subsection{The key local estimates}
We note that the mean curvature appearing below
refers to the viscosity mean curvature; at smooth points this of course coincides with the usual mean
curvature.  

Our first main estimate gives curvature control on a whole parabolic ball, from a mean 
curvature bound at a single point.

\begin{theorem}[Curvature estimate]\label{thm-intro_local_curvature_bounds}
For all $\al>0$ there exist $\rho=\rho(\al)>0$ and $C_\ell=C_\ell(\al)<\infty$
$(\ell=0,1,2,\ldots)$ with the following property.
If $\K$ is an $\al$-Andrews flow in a parabolic ball $P(p,t,r)$ centered at a boundary point 
$p\in \D K_t$ with   $H(p,t)\leq r^{-1}$, then  $\K$ is smooth in the parabolic ball
$P(p,t,\rho r)$, and
\begin{equation}\label{eqn-intro_curvature_estimate}
 \sup_{P(p,t,\rho r)}\abs{\nabla^\ell A}\leq C_\ell r^{-(\ell+1)}\, .
\end{equation}
\end{theorem}

Theorem \ref{thm-intro_local_curvature_bounds} implies in particular that
boundary points with finite mean curvature are regular points of the flow.
It also ensures that sequences of $\al$-Andrews flows have subsequences that
converge locally to smooth mean curvature flows, provided we normalize the 
mean curvature at a single point.
In particular, it recovers the gradient estimate of Huisken-Sinestrari \cite[Thm. 1.6]{huisken-sinestrari3}, see Corollary \ref{cor_HS_gradient_estimate}.

Our next result is the crucial convexity estimate.

\begin{theorem}[Convexity estimate]\label{thm-intro_convexity_estimate}
For all $\eps>0$, $\al >0$, there exists $\eta=\eta(\eps,\al)<\infty$ with the following property.
If $\K$ is an $\al $-Andrews flow in a parabolic ball  $P(p,t,\eta\, r)$ centered at a boundary point 
$p\in \D K_t$ with $H(p,t)\leq r^{-1}$, 
then
\begin{equation}
\la_1(p,t)\geq -\eps r^{-1}.
\end{equation}  
\end{theorem}

Theorem \ref{thm-intro_convexity_estimate} says that a boundary point $(p,t)$ in an
$\al$-Andrews flow has  almost positive definite
second fundamental form as long as the flow has had a chance to evolve
over a portion of spacetime which is 
large compared with the scale given by $H(p,t)$.
In particular, it recovers the convexity estimates of Huisken-Sinestrari \cite[Thm. 1.4]{huisken-sinestrari3} and White \cite[Thm. 8]{white_nature}, see Corollary \ref{cor-intro_convexity_estimates} and Corollary \ref{cor-intro_limit_flows_convex}.
It also implies that ancient globally defined $\al$-Andrews flows are convex, see Corollary \ref{cor-intro_ancient}. The convexity estimate, like the Hamilton-Ivey estimate in the Ricci flow \cite{Hamilton_survey}, plays a key role in developing the global theory.

\subsection{Global theory}
We start our treatment of the global theory by proving the following global convergence theorem.

\begin{theorem}[Global convergence]\label{thm-intro_h_gives_global_convergence}
Let $\{\K^j\}$ be a sequence of $\al $-Andrews flows, $\{(p_j,t_j)\in \D \K^j\}$
be a sequence of boundary points, and $\{r_j\}\subset (0,\infty)$ be a 
sequence of scales.
Suppose that for all $\eta<\infty$, for $j$ sufficiently large
 $\K^j$ is defined in $P(p_j,t_j,\eta \,r_j)$, and $H(p_j,t_j)\leq r_j^{-1}$.

Then, after passing to a subsequence, the flow $\hat \K^j$ obtained from $\K^j$
by the rescaling $(p,t)\mapsto (r_j^{-1}(p-p_j),r_j^{-2}(t-t_j))$ converges smoothly:
\begin{align}
\qquad\qquad\qquad \hat \K^j\ra\hat\K^\infty \qquad\qquad\qquad C^\infty_{\textrm{loc}} \,\, \textrm{on}\,\, \R^N\times(-\infty,0].
\end{align}
The limit $\hat\K^\infty$ is an $\al$-Andrews flow with convex time slices.
\end{theorem}

Theorem \ref{thm-intro_h_gives_global_convergence} allows us to pass to a
subsequence which converges globally and smoothly on backwards spacetime, after normalizing the mean curvature at a single point.
In fact, we can even extend the convergence
$\hat\K^j\ra\hat\K^\infty$
to the time interval $(-\infty,T)$, for any $T$ up to
the blowup time of $\hat\K^\infty$, provided $\{\K^j\}$ is defined on the
appropriate time intervals. The global convegence theorem is a powerful tool and indeed has many consequences; we will discuss some of them now.

The first application is in fact a variant of the global convergence theorem that works even without having
a priori bounds on the mean curvature at any point; this need arises when analyzing the formation of singularities, and the typical situation is as follows: Given an $\al$-Andrews flow 
$\K$, we consider a {\em blow-up sequence} $\K^j$ obtained by applying
a sequence of parabolic rescalings $\{(p,t)\mapsto (\la_j(p-p_j),\la_j^2(t-t_j))\}$
where $\{\la_j\}$ is any sequence of scaling factors tending to $\infty$.  After passing to a subsequence, we can always pass to a Hausdoff limit $\K^j\to \K$, called a {\em limit flow} (in the special case when the basepoint is fixed, it is called a \emph{tangent flow}).
Note that (assuming the basepoints $(p_j,t_j)$ don't hit the boundary of the domain of definition), the limit flow is defined on entire $\R^{N,1}$, and that it is an $\al$-Andrews flow (by definition).
A key question is then to investigate the structure all limit flows. Even more generally, one can investigate the structure of all ancient solutions.

\begin{theorem}[Structure of ancient $\al$-Andrews flows]
\label{thm-intro_smooth_convex_til_extinct}
Let $\K$ be an ancient $\alpha$-Andrews flow defined on $\R^N\times (-\infty,T_0)$ (typically $T_0=\infty$, but we allow $T_0<\infty$ as well),
and let $T\in (-\infty,T_0]$ be the extinction time of $\K$, 
i.e. the supremum of all $t$ with $K_t\neq\emptyset$. Then:
\begin{enumerate}
\item $\K\cap \{t<T\}$ is smooth. In fact, there exists a function $\ol{H}$ depending only on the Andrews constant $\al$ such that whenever $\tau< T-t$, then
$
H(p,t)\leq \ol{H}(\tau,d(p,K_{t+\tau})).
$
\item $\K$ is either a static halfspace, or it has strictly positive mean curvature and sweeps out all space, i.e. $\bigcup_{t<T_0}K_t=\R^N$.
\item $\K$ has convex time slices.
\item If $T<T_0$, the final time-slice $K_T$ has dimension at most $N-2$.
\end{enumerate}
Furthermore, if $\K$ is backwardly self-similar, then it is either (i) a static
halfspace or (ii) a shrinking round sphere or cylinder. In particular, every ancient $\al$-Andrews flow has a blow-down limit (ancient soliton) that is equal to one of these self-similar solutions.
\end{theorem}

Theorem \ref{thm-intro_smooth_convex_til_extinct} says in particular that ancient $\al$-Andrews flows are 
smooth and convex until they become extinct. Also, as was already implicit in the discussion above, given any $\al$-Andrews flow $\K$ we can apply
Theorem \ref{thm-intro_smooth_convex_til_extinct}
and Theorem \ref{thm-intro_h_gives_global_convergence} to its blow-up sequences
to recover the main results from \cite{white_nature}, which are White's structure
theorem for limit flows, and White's corollary about normalized limits.
Essentially, the assertion is that if $\K$ is a limit flow, then the convergence $\K^j\to \K$ is smooth away from the extinction, and this includes in particular normalized limits, see Corollary \ref{cor_white_limits} and Corollary \ref{cor_white_normalized_limits} for the precise statements.
Some typical examples for limit flows are the shrinking round cylinders 
$\R^{j}\times B^{N-j}$ ($T=0$, $K_0=\R^{j}$), translating solitons like the 
bowl ($T=\infty$), and static halfspaces ($T=\infty$). 

Theorem \ref{thm-intro_h_gives_global_convergence} and Theorem 
\ref{thm-intro_smooth_convex_til_extinct} naturally come along with their companion curvature 
estimates, Corollary \ref{thm-more_global_parabolic_bounds} and Corollary 
\ref{cor_curv_est_3}, respectively; see also the Harnack inequality in Corollary \ref{cor-intro_harnack}.  
 
As another application of Theorem \ref{thm-intro_smooth_convex_til_extinct},
we obtain a new proof, also entirely without using Brakke flows (like everywhere else in the present paper),
of White's partial regularity theorem \cite{white_size}.

\begin{theorem}[Partial regularity theorem]\label{thm_intro_partial_regularity}
For any $\al$-Andrews flow,
the parabolic Hausdorff dimension of the singular set is at most $N-2$.
\end{theorem}
We recall  that the {\em parabolic Hausdorff dimension} refers to the Hausdorff dimension
with respect to the parabolic metric on spacetime
$d((x_1,t_1),(x_2,t_2))=\max({|x_1-x_2|,|t_1-t_2|^{\frac12}})$.

\begin{remark}
In fact, using quantitative stratification instead of standard stratification 
and dimension reduction, one can strengthen the Hausdorff estimate to a Minkowski estimate and obtain estimates for the regularity-scale, etc; see Cheeger-Haslhofer-Naber \cite{CHN}.
\end{remark}

\subsection{$k$-convexity and the cylindrical estimate}
Several results have refinements in the case of $k$-convex flows.
Special instances are the convex case with Huisken's classical result
\cite{Huisken_convex},
the $2$-convex case of Huisken-Sinestrari
\cite{huisken-sinestrari3}, and 
the general mean convex case.
We  recall that a smooth normally oriented
hypersurface $M\subset \R^N$ is {\em $k$-convex} 
or {\em strictly $k$-convex} if 
$\lambda_1+\ldots+\lambda_k\geq 0$ or 
$\lambda_1+\ldots+\lambda_k> 0$ respectively,
where $\lambda_1\leq\ldots\leq\lambda_{N-1}$ 
denote the principal curvatures.   
Note that a cylinder $\R^j\times S^{N-1-j}\subset\R^N$ is strictly
$k$-convex if and only if  $j<k$.

The maximum principle
implies
that a compact smooth mean curvature flow with (strictly) $k$-convex initial condition remains 
(strictly) $k$-convex,
see Appendix \ref{app_preserved_curv}. In fact, after waiting a short time, by the strict maximum principle we can assume that
\begin{equation}
\label{eqn_beta_k-convexity}
\lambda_1+\ldots+\lambda_k\geq \be H,
\end{equation}
for some $\be>0$, which is again preserved along the flow.
Following the elliptic regularization approach of White \cite{white_subsequent},
it was shown in \cite{CHN} that this uniform $k$-convexity is also preserved beyond the first singular time, i.e. (\ref{eqn_beta_k-convexity})
holds at any smooth point $p\in \D K_t$. For convenience of the reader, we explain this again in Theorem \ref{thm_ell_reg}.

The uniform $k$-convexity yields further restrictions on the possible backwardly selfsimilar solutions that can show up in Theorem \ref{thm-intro_smooth_convex_til_extinct}. 

\textbf{Refinement of Theorem \ref{thm-intro_smooth_convex_til_extinct} for $k$-convex flows.}
\textit{If in addition (\ref{eqn_beta_k-convexity}) holds at all smooth points, then in the statement of Theorem \ref{thm-intro_smooth_convex_til_extinct} out of the cylinders $\R^j\times B^{N-j}$ only the ones with $j<k$ can arise.}

As a consequence, the Hausdorff dimension estimate in Theorem \ref{thm_intro_partial_regularity} can also be refined.

\textbf{Refinement of Theorem \ref{thm_intro_partial_regularity} for $k$-convex flows.}
\textit{If $\K$ is an $\al$-Andrews flow such that (\ref{eqn_beta_k-convexity}) holds at all smooth points, then the
parabolic Hausdorff dimension of the singular set is at most $k-1$.}

\begin{remark}
This in turn leads to refined Minkowski dimension estimates and $L^p$-estimates, see Cheeger-Haslhofer-Naber \cite{CHN}.
\end{remark}

Finally, we discuss the cylindrical estimate
which says,  roughly 
speaking, that near a boundary point in a uniformly $k$-convex flow, 
either the flow is
uniformly  $(k-1)$-convex or it is close to a shrinking round $(k-1)$-cylinder
$\R^{k-1}\times B^{N-(k-1)}$,
provided the
flow exists in a subset of backward spacetime which is large compared to the scale given by the mean curvature. 
To state this precisely, we say that an 
$\alpha$-Andrews flow is {\em $\varepsilon$-close to a shrinking round $j$-cylinder
(or cylindrical domain) 
$\mathbb{R}^j\times B^{N-j}$ near $(\bar p,\bar t)$}, if after applying
the parabolic rescaling 
$$
(p,t)\mapsto (H^{-1}(\bar p,\bar t)(p-\bar p),H^{-2}(\bar p,\bar t)(t-\bar t))
$$ 
and a rotation it becomes $\varepsilon$-close in the 
$C^{\lfloor 1/\varepsilon\rfloor}$-norm 
on $P(0,0,1/\varepsilon)$ to the standard shrinking $j$-cylinder with $H(0,0)=1$.

\begin{theorem}[Cylindrical estimate]\label{thm-intro_cylindrical}
For all $\varepsilon,\alpha,\beta>0$ there exists $\delta=\delta(\varepsilon,\alpha,\beta)>0$ 
such that the following holds.

Let $\K$ be an $\alpha$-Andrews flow that is uniformly $k$-convex in the sense
that 
$\lambda_1+\ldots + \lambda_k\geq \beta H$ at every smooth boundary point,
 and suppose $p\in\D K_t$ is a boundary point 
such that
$\K$ is defined in $P(p,t,\delta^{-1}H^{-1}(p,t))$. If 
$$
\frac{\lambda_1+\ldots+\lambda_{k-1}}{H}(p,t)<\delta\,, 
$$
then  $\K$ is $\varepsilon$-close to a shrinking round $(k-1)$-cylinder
$\mathbb{R}^{k-1}\times B^{N-(k-1)}$ near $(p,t)$.
\end{theorem}

We emphasize that Theorem \ref{thm-intro_cylindrical} holds for any $k$. It says in particular $\frac{|\lambda_i-\lambda_j|}{H}(p,t)<\varepsilon$ for all $i,j\geq k$. For $k=2$ and $N\geq 4$ this is essentially the statement of Huisken-Sinestrari \cite[Thm 1.5]{huisken-sinestrari3}, for $k=3$ they proved a similar result recently \cite{Huisken_privatecommunication}.

\begin{remark}
Theorem \ref{thm-intro_h_gives_global_convergence} quickly leads to many further structural results about $k$-convex $\alpha$-Andrews flows and the corresponding ancient solutions that arise after blowup. In particular, these ancient solutions have asymptotic curvature ratio equal to infinity, asymptotic volume ratio equal to zero and asymptotically split off a factor $\R^j$ for some $j < k$. Also, curvature control is equivalent to volume control.
\end{remark}

\subsection{Outline of proofs}
We now give a glimpse of the proofs.

\subsubsection*{Halfspace convergence}
The key for our  short treatment of the local theory is our halfspace convergence result, Theorem \ref{thm-halfspace_smooth_convergence}.
This result says, roughly speaking, that if an $\alpha$-Andrews flow is weakly close to some halfspace at some time, then it is strongly close to that halfspace both 
forward and backward in time.
To prove this, we first use
comparison with spheres and the Andrews condition to show that the flow is 
Hausdorff close to the halfspace both forward and backward in time.
Then, using the one-sided minimization property of mean convex flows and the local regularity theorem, we argue that it must in fact be strongly close. The proof takes less than a page, and vividly illustrates the efficiency of the 
$\al$-Andrews condition in combination with the elementary local regularity theorem of White \cite{white_regularity}. To keep everything self-contained and for convenience of the reader, we give a half-a-page proof of the local regularity theorem in Appendix \ref{app_easy_brakke}.

\subsubsection*{Curvature estimate and convexity estimate}
Arguing by contradiction we show that the halfspace convergence theorem quickly implies (and is in fact equivalent to) the curvature estimate. 
To prove the convexity estimate we consider the rigidity case of the maximum principle for $\frac{\lambda_1}{H}$, as in \cite{white_nature}.
However, while White had to give a very sophisticated argument to pass to a smooth local limit, in our proof  -- thanks to the halfspace convergence theorem -- everything is quite simple.

\subsubsection*{Global convergence}
The main claim to be proven is that the curvature of $\D \hat K^j_t$
in $B(0,R)$ is bounded by a function of $R$, for large $j$.  We prove this by contradiction
using a scheme of proof somewhat 
similar to the one in Perelman's proof of the canonical neighborhood theorem \cite[Sec. 12]{perelman_entropy}.
Roughly speaking, the crux of the argument is as follows.
We look at the supremal radius $R_0$ 
where such a curvature bound holds,  which then allows us to pass to a smooth
limit in the open ball $B(0,R_0)$.  This smooth limit will be convex.  
We then examine the structure of this
convex limit.  By applying the equality case of the maximum principle for
$\frac{\la_1}{H}$ and the Andrews condition, we argue that its mean
curvature remains controlled near  the sphere $S(0,R_0)$;
this leads to a contradiction with the choice of $R_0$. 
The actual argument is  somewhat more complicated due to the possibility
that the intersections $\hat K^j_t\cap B(0,R_0)$ might have more than one connected
component.

\subsubsection*{The structure theorem}
Part (1) is based on the following idea:
If we can find points with controlled mean curvature, then we can apply the global convergence theorem; if we cannot find such points, then we argue that the flow has to clear out very quickly.
The remaining assertions easily follow from the results established previously combined with a classical argument of Huisken \cite[Sec. 5]{Huisken_local_global}.

\subsubsection*{The cylindrical estimate}
If the cylindrical estimate failed, 
using the global convergence theorem we could pass to a limit that splits, 
$K_t=\R^{k-1}\times N_t$, but that is not a shrinking round cylinder.
However, since $N_t$ is an ancient $\alpha$-Andrews flow with $\lambda_1/H\geq \beta$ 
we argue that $\D N_t$ must be a shrinking round sphere; this gives the contradiction.

\subsubsection*{Approximation by smooth $\al$-Andrews flows} 
We consider any level set flow $\{K_t\}$ with smooth $\al$-Andrews initial condition and prove that the stabilized flow $\{K_t\times\R\}$ can be approximated by a smooth family of flows satisfying the Andrews condition.
Our proof is based on the elliptic regularization approach from \cite[Sec. 7]{evans-spruck},
\cite{white_subsequent}, and an adaption of the argument
by Andrews-Langford-McCoy \cite{Andrews_Langford_McCoy}.

\subsubsection*{The partial regularity theorem}
The partial regularity theorem follows from our structure theorem for ancient solutions and Huisken's monotonicity formula \cite{Huisken_monotonicity} (see also Appendix \ref{app_huisken_monotonicity}). Like everywhere else in the present paper, we do not need the notion of Brakke flows.

\subsection{Organization of the paper}
In Section \ref{sec-halfspace_and_consequences}, we prove the halfspace convergence theorem (Theorem \ref{thm-halfspace_smooth_convergence}),
the curvature estimate (Theorem \ref{thm-intro_local_curvature_bounds}),
and the convexity estimate (Theorem \ref{thm-intro_convexity_estimate}).
In Section \ref{sec-global_convergence_and_consequences}, we prove the global convergence theorem (Theorem \ref{thm-intro_h_gives_global_convergence}),
the structure theorem for ancient $\al$-Andrews flows (Theorem \ref{thm-intro_smooth_convex_til_extinct}),
and the cylindrical estimate (Theorem \ref{thm-intro_cylindrical}).
In Section \ref{sec_structure_regularity}, we prove the theorem about approximation by smooth flows satisfying the Andrews condition (Theorem \ref{thm_ell_reg}),
and the partial regularity theorem (Theorem \ref{thm_intro_partial_regularity}).
In Appendix \ref{app_preserved_curv}, \ref{app_huisken_monotonicity}, and \ref{app_easy_brakke} we explain some background material.

\emph{Smoothness and admissibility.} 
In Section \ref{sec-halfspace_and_consequences} and \ref{sec-global_convergence_and_consequences} we give the proofs in the smooth setting. Once the theorems are established in the smooth setting, it will be easy to obtain them for general $\al$-Andrews flows by approximation, see Section \ref{sec_structure_regularity}.
Also, some arguments can be shortened by imposing the technical assumption that the parabolic balls in the statement are admissible, see Section \ref{sec-halfspace_and_consequences} and Appendix \ref{app_remove_admissibility}.

\noindent\textbf{Acknowledgments.} We thank Brian White and Gerhard Huisken for helpful discussions about their work. BK has been supported by NSF Grants DMS-1007508 and DMS-1105656.

\section{Halfspace convergence and consequences}\label{sec-halfspace_and_consequences}

\subsection{Halfspace convergence}\label{subsec_halfspace_conv}
We begin with the following halfspace convergence result. This result is the key tool for our  short proofs of 
the curvature estimate and the convexity estimate.

\begin{theorem}[Halfspace convergence]
\label{thm-halfspace_smooth_convergence}
Suppose $T_0\geq 0$, and $\{\K^j\}$  is a sequence  of $\al$-Andrews flows such that
\begin{enumerate}
\item For every $R<\infty$, the flow $\K^j$ is defined in $P(0,T_0,R)$, for $j$ sufficiently large.
\item The origin $0\in\R^N$ lies in $\D K^j_0$ for every $j$.
\item Every compact
subset of the lower halfspace $\{x_N< 0\}$ is contained in the time zero slice 
$K^j_0$, for $j$ sufficiently large.
\end{enumerate}
Then $\K^j$ converges smoothly on compact 
subsets of $\R^N\times (-\infty,T_0]$  to the static halfspace $\{x_N\leq 0\}\times (-\infty,T_0]$.
\end{theorem}

Since this allows us to give a shorter proof and since this seems good enough for practically all applications,
we will temporarily (or more precisely until the end of Section \ref{subsec_curv_est}) replace the assumption (1) by the following slightly stronger \emph{admissibility assumption}:

\begin{enumerate}[(1')]
 \item For every $R<\infty$, the flow $\K^j$ is defined in $P(0,T_0,R)$ and some time slice $K^j_{t_j}$ contains $B(0,R)$, for $j$ sufficiently large.
\end{enumerate}

\begin{remark}
 The case $t_j\leq T_0-R^2$ is of course allowed. In fact, it follows from the assertion of the theorem that $t_j\to -\infty$.
\end{remark}

\begin{remark}\label{rem_admiss_always_holds}
Assumption (1') is satisfied for every blowup sequence.
\end{remark}

\begin{proof}[Proof of Theorem \ref{thm-halfspace_smooth_convergence} (smooth, admissible case)]

We begin by proving convergence to a halfspace in a weak sense:

\begin{claim}\label{claim_halfspace} The sequence of mean curvature flows $\{\K^j\}$ converges in the
pointed Hausdorff topology to a static halfspace in $\R^N\times (-\infty,T_0]$, and similarly
for their complements. 
\end{claim}

\begin{proof}[Proof of Claim \ref{claim_halfspace}]
For $R \in (0,\infty)$, and  $d\in \R$ let 
$$
\bar B_R^d= \ol{B((-R+d)e_N,R)}\,,
$$
so $\bar B_R^d$ is  the closed $R$-ball tangent to
the horizontal 
hyperplane $\{ x_N = d \}$   at the point $d\,e_N$.

When $R$ is large and $d>0$, 
it will take time approximately $dR$ for $\bar B_R^d$ to leave the upper 
halfspace $\{x_N>0\}$.  Since $0\in \D K^j_0$ for all $j$, it follows
that $\bar B_R^d$ cannot be contained in the interior of $K^j_t$
for any $t \in [-T,0]$, where $T \simeq dR$.   
By condition (3), for large $j$ 
we can
find $d_j\leq d$
such that $\bar B_R^{d_j}$ has interior contact with $K^j_t$ at 
some point $q_j$, where $\langle q_j,e_N\rangle < d$, $\|q_j\|\lesssim \sqrt{dR}$,
and moreover
$\liminf_{j\ra\infty} \langle q_j,e_N\rangle\geq 0$.  Hence the mean 
curvature satisfies $H(q_j,t)\leq \frac{N-1}{R}$.  Since $K^j_t$ satisfies 
the $\al$-Andrews condition, 
 there is a closed ball $\bar B_j$
with radius at least $\frac{\al R}{N-1}$ making  exterior contact with $K^j_0$
at $q_j$.
By a simple geometric calculation, this 
implies that $K^j_t$ has  height  $\lesssim \frac{d}{\al} $ in the ball
$B(0,R')$ where $R'$ is comparable  to $\sqrt{dR}$.  As $d$ and $R$ are 
arbitrary, this implies that for any $T>0$,  and any compact subset $Y\subset\{x_N>0\}$, 
 for large $j$ the time slice
$K^j_t$ is disjoint from $Y$, for all 
$t \geq -T$.   Likewise, for any $T>0$ and any
compact subset $Y\subset \{x_N<0\}$, 
the time slice
$K^j_t$ contains $Y$ for all 
$t \in[-T, T_0]$, and large $j$,
because $K^j_{-T}$ will contain a ball whose forward evolution under
MCF contains $Y$ at any time $t\in [-T, T_0]$.
\end{proof}

By admissibility, applying 
the one-sided minimization result from \cite[3.5]{white_size} as in \cite[3.9]{white_size},
we get 
 for every $\eps>0$, every time $t\leq T_0$ and every ball $B(x,r)$ centered
on the hyperplane $\{x_N=0\}$, that 
\begin{equation}\label{eqn_densitybound}
|\D K_t^j \cap B(x,r)| \leq (1+\eps)\omega_n r^n\,,
\end{equation}
for $j$ large
enough. In our smooth setting, this actually is completely elementary, see Remark \ref{rem_one_sided_minimization}.
Then, by the easy smooth version of Brakke's local regularity theorem (see Appendix \ref{app_easy_brakke}),
we have smooth convergence to a static halfspace.
\end{proof}

\begin{remark}
\label{rem_one_sided_minimization}
For convenience of the reader, let us give an elementary derivation of the one-sided minimization property and of the density bound (\ref{eqn_densitybound}).
In general, if $U\subset \R^N$ is an open set and $\{K_{t'}\subset U\}_{t'\leq t}$
is a smooth family of mean convex domains 
such that $\{\D K_{t'}\}$ foliates $U\setminus\Int(K_t)$, then $K_t$ has the following one-sided minimization property:
If $K'\supseteq K_t$ is a closed domain which agrees with $K_t$ outside a 
compact smooth domain $V\subset U$, then
$$
 |\D K_t\cap V|\leq |\D K'\cap V|.
$$
To see this, let $\nu$ be the vector field in $U\setminus \Int(K_t)$ defined by the 
outward unit normals of the foliation. Since $\Div\nu=H\geq 0$ we obtain   
$$
|\D K'\cap V|-|\D K_t\cap V|\geq \int_{\D K'\cap V}\langle\nu,\nu_{\D K'}\rangle
-\int_{\D K_t\cap V}\langle\nu,\nu_{\D K_t}\rangle
$$
$$
=\int_{(K'\setminus K_t)\cap V}\Div\nu\geq 0.
$$
Now in our situation, one can take as comparison domain $K'=K_t^j\cup (\bar{B}(x,r)\cap \{x_N\leq \delta\})$ for $\delta>0$ small, and this gives (\ref{eqn_densitybound}).
\end{remark}

\subsection{Proof of the curvature estimate}\label{subsec_curv_est}

We can now give a short proof of the curvature estimate (Theorem \ref{thm-intro_local_curvature_bounds}), assuming smoothness and admissibility.
For clarity, this means that we prove the following:

\begin{theorem18p}
For all $\al>0$ there exist $\rho=\rho(\al)>0$ and $C_\ell=C_\ell(\al)<\infty$
$(\ell=0,1,2,\ldots)$ with the following property. If $\K$ is a smooth $\al$-Andrews flow, and $P(p,t,r)$ is a parabolic ball centered at a boundary point $p\in \D K_t$, such that $\K$ is defined in $P(p,t,r)$, some time slice $K_{\bar t}$ contains $B(p,r)$, and $H(p,t)\leq r^{-1}$, then
\begin{equation}\label{eqn-sec2_curvature_estimate}
 \sup_{P(p,t,\rho r)}\abs{\nabla^\ell A}\leq C_\ell r^{-(\ell+1)}\, .
\end{equation}
\end{theorem18p}

\begin{proof}
We will first show that there exists a $\rho'>0$ such that the estimate (\ref{eqn-sec2_curvature_estimate}) holds for $\ell=0$ with $C_0=\frac{1}{\rho'}$.

Suppose this doesn't hold. Then there are sequences of $\al$-Andrews flows $\{\K^j\}$, boundary points $\{p_j\in \D K_{t_j}\}$ and scales $\{r_j\}$,
such that $\K^j$ is defined in $P(p_j,t_j,r_j)$, some time slice contains $B(p_j,r_j)$, and $H(p_j,t_j)\leq r_j^{-1}$, but
$\sup_{P(p_j,t_j,j^{-1}r_j)}|A|\geq jr_j^{-1}$.

After parabolically rescaling by $(j^{-1}r_j)^{-1}$ and applying an isometry, we
obtain a new sequence $\{\hat \K^j\}$ of $\al$-Andrews flows such that:
\begin{enumerate}
\renewcommand{\theenumi}{\alph{enumi}}
\item $\hat \K^j$ is defined in $P(0,0,j)$ and some time slice contains $B(0,j)$.
\item $0\in \D \hat K^j_0$ and the outward normal of $\hat K^j_0$ at $(0,0)$ 
is $e_N$.
\item $H_{\D\hat K^j_0}(0,0)\ra 0$ as $j\ra\infty$.
\item $\sup_{P(0,0,1)}|A|\geq 1$.
\end{enumerate}
By (a), (b), (c) and the $\al$-Andrews condition, $\{\hat\K^j\}$ satisfies the assumptions (1'), (2) and (3) of 
Theorem \ref{thm-halfspace_smooth_convergence}, and hence it converges
smoothly on compact subsets of spacetime to a static halfspace; this
contradicts (d).

Finally, by standard derivative estimates (see e.g. \cite[Prop. 3.22]{Ecker_book}), we get uniform bounds on all scale invariant
derivatives
of $A$ in $P(p,t,\frac{\rho}{2} r)$. Setting $\rho=\frac{\rho'}{2}$, the theorem follows.
\end{proof}

As an immediate consequence of Theorem 1.8', we obtain the gradient estimate of Huisken-Sinestrari \cite[Thm. 1.6]{huisken-sinestrari3}:

\begin{corollary}\label{cor_HS_gradient_estimate}
Suppose $\K$ is a smooth mean convex flow, where the initial time slice is
compact. Then
\begin{equation}
 \abs{\nabla A}\leq CH^2
\end{equation}
for a constant $C<\infty$ depending only on the initial time slice.
\end{corollary}

\begin{remark}
\label{rem-forward_extension}
One may obtain a variant of the curvature estimate by considering 
flows which are defined in $B(p,r)\times (t-r^2,t+\tau r^2]$ for some
fixed $\tau>0$, in which case the curvature bound holds in a suitable parabolic
region extending forward in time. The proof is similar.
\end{remark}

\begin{remark}\label{rem_remove_admiss}
The assumption that some time slice $K_{\bar t}$ contains $B(p,r)$ actually can be removed again everywhere, see Appendix \ref{app_remove_admissibility}.
\end{remark}

\begin{remark}\label{rem_dropping}
Theorem \ref{thm-intro_local_curvature_bounds} still holds for flows satisfying the Andrews condition, where one allows that at finitely many times some connected components are discarded. Since dropping components has the good sign in Huisken's monotonicity inequality, the same proof applies.
\end{remark}

\subsection{Proof of the convexity estimate}

Based on the curvature estimate, we can now give a short proof of the convexity estimate.

\begin{proof}[Proof of Theorem \ref{thm-intro_convexity_estimate} (smooth case)]
Our proof, like Brian White's proof of the convexity estimate in \cite{white_nature},
is based on rigidity in the equality case of the maximum principle for the quantity
$\frac{\la_1}{H}$.

Fix $\al $.  The $\al$-Andrews condition implies that the assertion holds for $\eps = \frac1\al $. 
Let $\eps_0\leq \frac1\al $ be the infimum of the $\eps$'s for which it holds, 
and suppose $\eps_0>0$.

It follows that there is a sequence $\{\K^j\}$ of $\al$-Andrews flows, 
where for all $j$, $(0,0)\in \D \K^j$, $H(0,0)\leq 1$ and $\K^j$ is defined in 
$P(0,0,j)$, but ${\la_1}(0,0)\to -\eps_0$ as $j\ra \infty$.
After passing to a subsequence, 
$\{\K^j\}$  converges smoothly to a mean curvature flow $\K^\infty$ in the parabolic
ball $P(0,0,\rho)$, where $\rho=\rho(\al)$ is the quantity from Theorem
\ref{thm-intro_local_curvature_bounds} (respectively the quantity of Theorem 1.8' for readers only interested in the admissible version of the convexity estimate). Then for $\K^\infty$ we have $\lambda_1(0,0)=-\eps_0$ and thus $H(0,0)=1$.

By continuity $H>\frac12$
in  $P(0,0,r)$ for some $r\in (0,\rho)$.  Furthermore we have $\frac{\la_1}{H}\geq -\eps_0$
everywhere in $P(0,0,r)$.  This is because   every $(p,t)\in \D\K^\infty\cap P(0,0,r)$
is a limit of a sequence $\{(p_j,t_j)\in  \D \K^j\}$ of boundary 
points, and  for every $\eps > \eps_0$, if 
$\eta=\eta(\eps,\al)$, then for 
large $j$, $\K^j$ is defined in $P(p_j,t_j,\eta H^{-1}(p_j,t_j))$,
which implies that 
the ratio $\frac{\la_1}{H}(p_j,t_j)$ is bounded below by $-\eps$. 
Thus, in the parabolic ball $P(0,0,r)$,  the ratio $\frac{\la_1}{H}$ attains a 
negative minimum $-\eps_0$ at $(0,0)$. Since $\lambda_1<0$ and $\lambda_{N-1}>0$ the Gauss curvature $K=\lambda_1\lambda_{N-1}$ is strictly negative.
However, by the equality case of the maximum principle (see also Appendix \ref{app_preserved_curv}), the hypersurface locally splits as a product and thus this Gauss curvature must vanish; a contradiction.
\end{proof}

As an immediate consequence of (the smooth, admissible version of) the convexity estimate, we obtain the original versions of the convexity estimates due to Huisken-Sinestrari \cite[Thm. 1.4]{huisken-sinestrari3} and White \cite[Thm. 8]{white_nature}:

\begin{corollary}\label{cor-intro_convexity_estimates}
Suppose $\K$ is a smooth mean convex flow, where the initial time slice is
compact. Then for all $\eps>0$ there is an $H_0<\infty$ such that
if $H(p,t)\geq H_0$ then $\frac{\la_1}{H}(p,t)\geq -\eps$.
\end{corollary}

\begin{corollary}\label{cor-intro_limit_flows_convex}
Every special limit flow of a mean convex flow has nonnegative second fundamental form at all its regular points.\footnote{For the moment we only obtain the statement for special limit flows, i.e. limit flows up to the first singular time. Later times will be discussed later.}
\end{corollary}

Another consequence is the convexity of ancient $\al$-Andrews flows. This is analogous to the theorem that ancient Ricci flows in dimension three have nonnegative curvature operator \cite{Hamilton_survey,chen_uniqueness}.

\begin{corollary}\label{cor-intro_ancient}
If $\{K_t\subset \R^N\}_{t\in(-\infty,T)}$ is an ancient $\al$-Andrews flow, then $K_t$ is convex for all $t\in (-\infty,T)$.
\end{corollary}

\begin{remark}\label{remark_sweepout}
It follows from an elementary argument based on the Andrews condition that $\{K_t\}$ either is a static halfspace or it sweeps out all space, i.e. $\bigcup_{t\in(-\infty,T)} K_t=\R^N$, c.f. Lemma \ref{lemma_admissibility}.
\end{remark}

\begin{proof}[Proof of Corollary \ref{cor-intro_ancient} (smooth case)]
By Theorem \ref{thm-intro_convexity_estimate}, the boundary $\D K_t$ has positive semidefinite second fundamental form for every $t$. Thus, picking any $p\in K_{T}$, the connected component $K_t^{p} \subset K_t$ containing $p$ is convex.
We claim that there are no other connected components, i.e. $K_t^{p}=K_t$. Indeed, suppose for any $R<\infty$ there was another component $K'_t$ in $B(p,R)$.
Going backward in time, such a complementary component $K'_{\bar t}$ would have to stay disjoint from our principal component $K^p_{\bar t}$, and thus $K'_{\bar t}$ would have to slow down. But then the Andrews condition would 
clear out our principal component $K^p_{\bar t}$; a contradiction. 
\end{proof}

\section{Global convergence and consequences}\label{sec-global_convergence_and_consequences}

\subsection{Proof of the global convergence theorem}

In this section we give the proof of the global convergence theorem. This is the only proof that is longer than a single page and may be skipped at first reading.

\subsubsection{Preliminaries about convex sets}\label{subsec-convex_sets}
Suppose $C\subset \R^N$ is a convex set and $p\in C$.
For every $\la\geq 1$ let $C_{p,\la}=\la(C-p)=\{\la(x-p)\mid x\in C\}$.  The tangent cone $\cone_p(C)$ of $C$ at $p$  is defined as
$
\ol{\cup_{\la\geq 1}C_{p,\la}}\,.
$
By convexity, $\{C_{p,\la}\}$ is nested, i.e. $\la_1\leq \la_2$ $\implies$
$C_{p,\la_1}\subset C_{p,\la_2}$.  The family $\{C_{p,\la}\}$ converges
to $\cone_p(C)$ in the pointed Hausdorff topology, and likewise for their complements.  If $C$ has nonempty interior, then so does $\cone_p(C)$,
and  every compact subset  of the interior of $\cone_p(C)$ 
is contained in  the interior of $C_{p,\la}$ for $\la$ sufficiently large.  

If $C\subset \R^N$ is a closed convex set with nonempty interior, then 
$C$ is a topological manifold with boundary, and $\D C$ is locally the graph of a Lipschitz function.
In particular, almost every $p\in \D C$ with respect to 
$(N-1)$-dimensional Hausdorff measure is a point of 
differentiability, and at any such point the tangent
cone $\cone_p(C)$ is a halfspace.

\subsubsection{Steps of the proof}
We recommend that the reader now recalls the outline of the proof from the introduction.
Our proof has seven steps.
Step 1 describes the setup for proving the key curvature bounds. The core of the argument is contained in Step 2 and 3. The curvature bounds are eventually obtained in Step 5, and essentially they are what is needed to conclude the argument in Step 7.
Step 4 and also some parts of Step 1, 6 and 7 deal with some technical issues caused by the a priori possibility of having more than one connected component.

\begin{proof}[Proof of Theorem \ref{thm-intro_h_gives_global_convergence} (smooth case)] 

Since the hypotheses and conclusions are invariant under parabolic rescaling, 
we may assume without loss of generality that $(p_j,t_j)=(0,0)$ and
$H(p_j,t_j)=H(0,0)\leq 1$ for all $j$.\\

\noindent{\em Step 1.  Setup for controlling $H(x)$ as a function of $d(x,0)$.}

For all $R\in (0,\infty)$, let $X^j_R$ be the connected component of 
$ K^j_0\cap B(0,R)$ containing $0\in \R^N$.  Let $R_0$
be the supremum of the numbers $R>0$ such that there is a $C=C(R)$
with $H\leq C$ in $\D X^j_R\subset B(0,R)$ for large $j$;
note that
$R_0>0$ by Theorem \ref{thm-intro_local_curvature_bounds}.

Suppose $R_0<\infty$.   After passing to a subsequence, we may assume that
\begin{equation}
\label{eqn-blowup_at_r0}
\lim_{j\ra\infty}\left(\sup_{\D X^j_R}H  \right)=\infty
\end{equation}
for all $R>R_0$.

By the definition of $R_0$ and 
Theorem \ref{thm-intro_local_curvature_bounds},
after passing to a subsequence,
we may assume that  $\{X^j_{R_0}\}$ converges smoothly on compact subsets
of $B(0,R_0)$ to a domain with smooth boundary $X^\infty\subset B(0,R_0)$.
In fact, for every $x\in X^\infty$, there is an $r>0$ such that 
the convergence $X^j_{R_0}\ra X^\infty$
extends to a backward parabolic ball $P(x,0,r)$. 

Let  $X\subset X^\infty$ be the connected component of $X^\infty$ containing
$0$.\footnote{Although $X^j_{R_0}$ is connected, 
since the convergence is only on compact
 subsets of $B(0,R_0)$, a priori $X^\infty$ might not be connected.}
Then $X$ is convex, because the second fundamental form of $\D X^\infty$
is positive semidefinite, by Theorem \ref{thm-intro_convexity_estimate}.
Hence the closure $\ol{X}\subset \ol{B(0,R_0)}$
is a compact convex set with nonempty interior.\\

\noindent{\em Step 2. Controlling the curvature of $\D X$ near the 
sphere $S(0,R_0)$:  For all 
$q\in\overline{\partial X}\cap S(0,R_0)$ we have 
$\inf\{H(x)d(x,q)| x\in\partial X\}=0$.}

Suppose there is a point  $q$ in the sphere $S(0,R_0)$ 
lying in the closure $\ol{\D X}\subset \ol{B(0,R_0)}$, such that
\begin{equation}
\label{eqn-dh_bounded_below}
\inf\{H(x)\,d(x,q)\mid x\in \D X\}>0\,.
\end{equation}
Let $X_1$ be the tangent cone of the convex set $\ol{X}$ at $q$.

Suppose $X_1$ were a halfspace.  Then  $X_1$ would coincide with the
halfspace $Y=\{y\in \R^N\mid \langle y,q\rangle \leq 0\}$.  Also, for any 
sequence $x_k\in \D X$ with $x_k\ra q$, and any $\Lambda<\infty$, by using the
fact that $\frac{1}{d(x_k,q)}(X-q)$  pointed Hausdorff converges to the halfspace $X_1=Y$,
for large $k$ we could find a point $x_k'\in B(q,2d(x_k,q))\cap \D X$ 
which has interior contact
with a ball of radius at least $\Lambda\, d(x_k,q)$, contradicting (\ref{eqn-dh_bounded_below}).
Therefore $X_1$ is not a halfspace.

Now we may choose  a point
$q_1\in \D X_1\setminus \{0\}$ on the 
boundary $\D X_1$ which lies in the interior of the halfspace $Y$, i.e. 
$q_1\in \D X_1\setminus
\D Y$, such that the tangent cone of $X_1$ at $q_1$ is a halfspace;
 see preliminaries.  

By Theorem \ref{thm-intro_local_curvature_bounds} the point $q_1$ is a smooth point of $X_1$, and  for some $r>0$ 
the intersection $X_1\cap B(q_1,r)$ can be extended
to a smooth mean curvature flow $\hat\K^\infty$ in a backward parabolic ball
$P(q_1,0,r)$, which is a smooth limit of rescalings $\{\hat{\K}^j\}$ of a subsequence of $\{\K^j\}$.
Theorem \ref{thm-intro_convexity_estimate} implies that $\hat \K^\infty$ has convex time slices, and (\ref{eqn-dh_bounded_below}) yields $H(q_1,0)>0$.

Recapping, we have a smooth convex mean curvature flow $\hat\K^\infty$ whose final time slice is part of a nonflat cone;
this contradicts the equality case of the maximum principle for $\frac{\lambda_1}{H}$, c.f. Appendix \ref{app_preserved_curv}.\\

\noindent{\em Step 3. Controlling the mean curvature of $K^j_0$ near $\ol{X}$: there exist
$\de>0$ and  $\bar H<\infty$ such that for large $j$
we have $H\leq\bar H$ in $\D K^j_0\cap N_\de(\ol{X})$,
where $N_\de(\ol{X})$ is the $\de$-neighborhood of $\ol{X}$ in $\R^N$.}

By compactness, it suffices to prove that every $q\in\ol{X}\cap S(0,R_0)$ has a neighborhood where the curvature is bounded.

If $q\in \ol{\D X}\cap S(0,R_0)$, by Step 2,
there is a sequence $\{x_k\}\subset \D X$ with
$x_k\ra q$ such that $H(x_k)d(x_k,q)\ra 0$.  
Putting $r_k=\frac{2}{\rho}d(x_k,q)$, where $\rho$ is the constant
from Theorem \ref{thm-intro_local_curvature_bounds}, we still have $H(x_k)r_k\ra 0$.
So for a fixed  sufficiently large $k$, by applying 
Theorem \ref{thm-intro_local_curvature_bounds} to the approximators $\K^j$ with
basepoints $p_j\in \D K^j_0$ converging to $x_k$, 
we obtain $r>0,C<\infty$ such that
for large $j$ we have $H\leq C$ in $ \D K^j_0\cap B(q,r)$.  

If $q\in (\ol{X}\cap S(0,R_0))
\setminus \ol{\D X}$, then there is a closed ball $\ol{B(x,\rho)}\subset \ol{X}$ tangent to
$\D \ol{X}$ at $q$.  Therefore, there is an $r\ll\rho$ such that for large $j$,
either  $B(q,r) \subset  K^j_0$ or $B(q,r)\cap \D K^j_0\neq\emptyset$
and there is a translate of $\ol{B(x,\rho)}$
with touches $\D K^j_0$ from the inside at some  point $y$ with $d(y,q)\ll\rho$,
 and hence
we get that  
 $\D K^j_0$ has $H\lesssim \rho^{-1}$ and has tangent space nearly parallel to 
 $T_qS(0,R_0)$
in $B(q,r)$, by Theorem \ref{thm-intro_local_curvature_bounds}.\\

\noindent {\em Step 4. If $R_1=R_0+\tau$ and $\tau\ll \bar H^{-1}$, then
$X^j_{R_1}\subset N_{2\tau}(\ol{X})$ for large $j$.}

If not, then 
for some large $j$, there is an $x\in X^j_{R_1}\setminus N_{2
\tau}(\ol{X})$.
By the definition of $X^j_{R_1}$, there is a path $\ga$ from $0$ to $x$
lying
in $K^j_0\cap B(0,R_1)$.  Without loss of generality we may assume
that 
$\ga \setminus \{x\}\subset N_{2\tau}(\ol{X})$, and $d(x,\ol{X})=2\tau$.
Let $\pi(x)\in \ol{X}$ be the point in $\ol{X}$ closest to $x$. Since $2\tau>\tau$ we must have $\pi(x)\in\ol{\D X}$.

Since $\D K^j_0$ has curvature bounded by $\bar{H}\ll \tau^{-1}$
in $N_\de(\ol{X})$,
after passing to a subsequence the intersection
$K^j_0\cap B(\pi(x),10\tau)$ will converge smoothly to a convex domain
with smooth boundary  $\tilde X\subset B(\pi(x),10\tau)$, which looks very
close to a halfspace. Note that $\tilde X$ equals $\ol{X}$ in the region where they are both defined.

If $\pi(x)\in \D X$, since $\pi(x)$ is the nearest point in $\ol{X}$,
the vector $x-\pi(x)$ must be a positive multiple of the outward unit normal $\nu$ of $\tilde X$ at $\pi(x)$; 
then the Andrews condition gives $x\notin \partial K^j_0$, a
contradiction.

If
$\pi(x)\in \ol{\D X}\setminus \D X$, an elementary geometric argument shows that we 
still have $\langle \frac{x-\pi(x)}{\|x-\pi(x)\|},\nu\rangle \geq c$ for some universal
constant $c>0$; then the Andrews condition gives the same contradiction as before.\\

\noindent{\em Step 5: Getting $R_0=\infty$ and universal curvature bounds $\bar H_R$.}

We have shown that for large $j$, we have $X^j_{R_1}\subset N_{2\tau}(\ol{X})
\subset N_\de(\ol{X})$, and $H\leq\bar H$ in $N_\de(\ol{X})$.  This contradicts
(\ref{eqn-blowup_at_r0}).  Therefore $R_0=\infty$.  

In fact, since our reasoning thus far applies to 
any sequence satisfying the hypotheses of the theorem, we may apply it
once again to see  that
for every $R<\infty$ there exists $\bar H_R<\infty$ which does not
depend on the sequence $\K^j$, such that (with the normalization
$H(0,0)\leq 1$)
\begin{equation}
\label{eqn-universal_h_bound}
\lim_{j\ra\infty}\left(\sup_{\D X^j_R}H  \right)\leq \bar H_R\, .
\end{equation}

\noindent{\em Step 6: Obtaining a global convex limit at time $0$.}

Now we may pass to a subsequence of $\{\K^j\}$
such that  $\{X^j_R\}$ has a smooth limit $X^\infty_R$  for all $R>0$, and 
let $X_R$ be the connected component of $X^\infty_R$ containing $0$.

Arguing as in Step 4,
we  conclude that  $X^j_{R}\subset X^j_{R+\tau}\subset N_{2\tau}(\ol{X}_R)$,
for $\tau>0$ sufficiently small, and large $j$.  
Letting  $\tau\ra 0$ gives  $X^\infty_R\subset X_R$.  Therefore
$X_R=X^\infty_R$ for all $R$. 

Suppose $R_1\leq R_2$.  Then since $X^j_{R_1}\subset X^j_{R_2}$ for all $j$,
we have  $X^\infty_{R_1}\subset X^\infty_{R_2}\cap B(0,R_1)$.  On the other hand,
if $x\in X^\infty_{R_2}\cap B(0,R_1)$ then by the convexity of $X^\infty_{R_2}$, 
the radial segment $\ol{0x}$ is contained in $X^\infty_{R_2}\cap B(0,R_1)$, and 
small perturbations of $\ol{0x}$ will yield paths in $X^j_{R_1}$, for large $j$,
showing that $x\in X^\infty_{R_1}$.  Thus $X^\infty_{R_2}\cap B(0,R_1)=
X^\infty_{R_1}$.  Hence 
the union $X^\infty=\cup_{R}\,X^\infty_R$ is a smooth, closed, convex domain of $\R^N$.\\

\noindent{\em Step 7: Global convergence to a convex $\al$-Andrews flow.}

For $t\leq 0$, let  $X^j_{R,t}$ be the connected component of $K^j_t\cap B(0,R)$
containing $0$.
By comparison with shrinking spheres we can find points at controlled distance with controlled mean curvature. Thus (\ref{eqn-universal_h_bound}) generalizes to
\begin{equation}
\label{eqn-universal_h_bound2}
\lim_{j\ra\infty}\left(\sup_{\D X^j_{R,t}}H  \right)\leq f(R,t)\, .
\end{equation}
for some continuous function $f$.

Applying Steps 1--6 at each nonnegative rational time, and passing to a subsequence, there are convex sets $K^\infty_t$
such that the domains $X^j_{R,t}$ converge smoothly to 
$K^\infty_t\cap B(0,R)$ as $j\ra\infty$, for all $R<\infty$ and all $t\in\Q\cap (-\infty,0]$.

Let $\K^\infty$ be the closure of  $\cup_{t\in \Q\cap(-\infty,0]}\; K^\infty_t\times \{t\}
\subset \R^N\times (-\infty,0]$. From (\ref{eqn-universal_h_bound2}) it follows that $\K^\infty$ is a smooth mean curvature flow and that the convergence is for all nonnegative times and not just the rational ones.

Arguing as in the proof of Corollary \ref{cor-intro_ancient} we see that potential other connected components of $K_t^j\cap B(0,R)$ are cleared out, and consequently $\K^j\to \K^\infty$ smoothly on compact subsets of $\R^N\times (-\infty,0]$.

Finally, applying Theorem \ref{thm-intro_local_curvature_bounds} and the Andrews condition to the approximators we see that $\K^\infty$ is either a static halfspace or has strictly positive mean curvature everywhere, and it follows that $\K^\infty$ is an $\al$-Andrews flow with convex time slices.
\end{proof}

\begin{remark}
Theorem \ref{thm-intro_h_gives_global_convergence} still holds in the more general setting of Remark \ref{rem_dropping}. In fact, by Step 7 of the proof all potential other components that might be discarded are eventually cleared out.
\end{remark}

The global convergence theorem immediately implies (and quickly follows from) its companion curvature estimate.

\begin{corollary}[Curvature estimate II]
\label{thm-more_global_parabolic_bounds}
For all $\al>0$ and $\Lambda<\infty$, there exist 
$\eta=\eta(\al,\Lambda)<\infty$ and $C_\ell=C_\ell(\al,\Lambda)<\infty$ $(\ell =0,1,2,\ldots)$ such that if $\K$ is an $\alpha$-Andrews flow
in a parabolic ball $P(p,t,\eta r)$ centered at a boundary point $p\in \D K_t$ with $H(p,t)\leq r^{-1}$, then $\K$ is smooth in $P(p,t,\Lambda r)$ and
\begin{equation}\label{glob_der_est}
\sup_{P(p,t,\Lambda r)}|\nabla^\ell A|\leq C_\ell r^{-(\ell+1)}\qquad (\ell=0,1,2,\ldots).
\end{equation}
\end{corollary}

As another immediate consequence of the global convergence theorem, we obtain a Harnack inequality for $\al$-Andrews flows:

\begin{corollary}[Harnack inequality]\label{cor-intro_harnack}
For all $\al>0$ and $\Lambda<\infty$, there exists
$\eta=\eta(\al,\Lambda)<\infty$ such that if $\K$ is an $\alpha$-Andrews flow
and $p\in \D K_t$ is a boundary point such that $\K$ is
defined in $P(p,t,\eta H^{-1}(p,t))$, then
\begin{equation}
\eta^{-1} H(p,t) \leq H(p',t') \leq \eta H(p,t)
\end{equation}
for all boundary points $(p',t')$ in $P(p,t,\Lambda H^{-1}(p,t))$.
\end{corollary}

\subsection{Proof of the structure theorem}

In this section, we prove the structure theorem for ancient $\al$-Andrews flows (Theorem \ref{thm-intro_smooth_convex_til_extinct}). For the proof we need the following  lemma.

\begin{lemma}[Speed limit lemma]\label{lemma_speed}
\label{lem-speed_limit}
Suppose $\{K_t\}_{t\in[t_0,t_1]}$ is a smooth mean convex flow in $B(p_1,r)$,
 $\D K_{t_0}\neq\emptyset$ and $p_1\in K_{t_1}$. Then there exists a boundary point 
 $p\in \D K_t$ $(t_0\leq t\leq  t_1)$
such that $H(p,t)\leq \frac{r}{t_1-t_0}$.
\end{lemma}
\begin{proof}
Let $f:[t_0,t_1]\ra [0,\infty)$ be defined by $f(s)=d(p_1,\D K_s)$.  
Since $\K$ is mean convex, $f$ is nonincreasing; also, if $\bar t\in [t_0,t_1)$
and $f(\bar t)=\bar r$, then
there is an inscribed $\bar r$-ball $\bar{B}(p_1,\bar r)\subset K_{\bar t}$, which gives a lower
bound $f^2(t)\geq \bar r^2-2(N-1)(t-\bar t)$ for $t\geq \bar t$.  Therefore $f$ is locally
Lipschitz on $[t_0,t_1)$.  
Since $f(t_1)-f(t_0)\geq-r$, there is a point of differentiability $t\in [t_0,t_1)$
such that $f'(t)\geq-\frac{r}{t_1-t_0}$.  If $p\in \D K_{t}$ realizes
the distance from $p_1$ to $\D K_{t}$, then $H(p,t)\leq -f'(t)\leq\frac{r}{t_1-t_0}$.
\end{proof}

\begin{proof}[Proof of Theorem \ref{thm-intro_smooth_convex_til_extinct} (smooth case)]
More precisely, with proof in the smooth case we mean that we assume the flow is smooth for $t<T$. We of course do allow singularities at the extinction time $t=T$.

\noindent (1) Given $\tau< T-t$, by Lemma \ref{lemma_speed} we can find a boundary point $(p',t')\in\D \K$ with  $\abs{p-p'}\leq d(p,K_{t+\tau})$, $t'\in [t,t+\tau]$ and $H(p',t')\leq\frac{d(p,K_{t+\tau})}{\tau}$. Then, Theorem \ref{thm-intro_h_gives_global_convergence} gives universal curvature bounds.

\noindent (2) If the mean curvature vanishes at some point, by Theorem \ref{thm-halfspace_smooth_convergence} and the Andrews condition the flow must be a static halfspace. If the flow is not a static halfspace, by Remark \ref{remark_sweepout} it must sweep out all space.

\noindent (3) By Corollary \ref{cor-intro_ancient} the time slices $K_t$ are convex for $t<T$.
Furthermore, if $T<T_0$ then $K_T=\cap_{t<T}K_t$, by monotonicity
of time slices (mean convexity) and the fact that $\K$ is a closed subset
of $\R^N\times (-\infty,T_0)$. Thus, the final time slice $K_T\subset \R^N$ is also closed and convex.

\noindent (4) Recall that the dimension of a closed convex set is an integer and equals the dimension of the smallest affine space where it is contained in.
We must rule out the cases $\dim K_T=N$ and $\dim K_T=N-1$. Suppose $\dim K_T=N$. Then $K_T$ has non empty interior; this contradicts the definition of extinction time.
Suppose $\dim K_T= N-1$. Then $K_T$ is contained in an $(N-1)$-plane $V$ and we can find an interior point of $K_T\cap V$ and thus an open disc $D\subset K_T\cap V$. 
Moving a tiny bit backwards in time and bending this disc slightly until it contacts
the boundary, we find a point with small mean curvature.  By the Andrews
condition and avoidance, this implies that $K_T$ contains a large ball, which is a contradiction.

Furthermore, the argument of Huisken \cite[Sec. 5]{Huisken_local_global} shows that
any backwardly self-similar $\al$-Andrews flow with $H>0$ must be a shrinking round sphere or cylinder,
provided we can justify Huisken's partial integration for the term $\int \abs{\nabla \frac{\abs{A}^2}{H^2}}^2 e^{-\abs{x}^2/2}$
without apriori assumptions on curvature and volume.
This was already explained in \cite[Proof of Thm. 10]{white_nature}, but for convenience of the reader let us reproduce the argument here: Recall first that the $t=-1/2$ slice of a backwardly selfsimilar solution satisfies
\begin{equation}
 H(x)=\langle x, \nu\rangle .
\end{equation}
Together with the convexity established in part (3), this shows that the curvature grows at most linearly,
\begin{equation}
 \abs{A}\leq H\leq \abs{x},
\end{equation}
and similarly for the derivatives. Also, by the one-sided minimization property (Remark \ref{rem_one_sided_minimization}) the volume growth is at most polynomial.
Thus, Huisken's partial integration is justified in our context.

If $\K$ is a ancient $\al$-Andrews flow, and $\lambda_j\to 0$, let $\K^j$ be the blowdown sequence obtained by parabolically recaling $(p,t)\to (\lambda_j p,\lambda_j^2 t)$.
By comparison with spheres and Lemma \ref{lemma_speed}, we can find points at controlled distance from the origin with controlled mean curvature. Thus, by Theorem \ref{thm-intro_h_gives_global_convergence} a subsequence converges to a limit $\K^\infty$, which by Husiken's monotonicity formula \cite{Huisken_monotonicity} must be self-similarly shrinking.
By the above, it must be either a plane, a cylinder, or a sphere.

Finally, observing that out of the cylinders $\R^j\times B^{N-j}$ only the ones with $j<k$ satisfy (\ref{eqn_beta_k-convexity}), gives the refinement in the $k$-convex case.
\end{proof}

Part (1) of the structure theorem immediately implies (and quickly follows from) its companion curvature estimate:

\begin{corollary}[Curvature estimate III]\label{cor_curv_est_3}
For all $\al>0$ and $\La<\infty$, there exist $\eta=\eta(\al,\La)<\infty$, $C=C(\al,\La)<\infty$
such that if $\K$ is an $\al$-Andrews flow, $p\in K_t$, $p'\in \D K_{t'}$, $t'<t$, $\frac{d^2(p,p')}{t-t'}<\La$, $\K$ is defined in $P(p,t,\eta\sqrt{t-t'})$ and some time slice
of $\K$ contains $B(p,\eta\sqrt{t-t'})$, then
$$
H(p',t')\leq \frac{C}{\sqrt{t-t'}}\,.
$$
\end{corollary}

\subsection{Proof of the cylindrical estimate}

In this section, we prove the cylindrical estimate for $k$-convex flows (Theorem \ref{thm-intro_cylindrical}). For the proof, we need the following lemma.

\begin{lemma}\label{lemma_sphere}
If $\{K_t\subset \R^N\}_{t\in(-\infty,T)}$ is an ancient $\al$-Andrews flow with $\frac{\lambda_1}{H}\geq\beta>0$, then it must be a round shrinking sphere.
\end{lemma}

\begin{proof}[Proof (smooth case)]
By Huisken's monotonicity formula \cite{Huisken_monotonicity} (see also Appendix \ref{app_huisken_monotonicity})
and Theorem \ref{thm-intro_smooth_convex_til_extinct} we can find an asymptotic soliton for $t\to -\infty$, which due to
$\frac{\lambda_1}{H}\geq \beta$ must be a round shrinking sphere.
Similarly, possibly after extending the flow, we can obtain a tangent flow at the singular time which also must be a round shrinking sphere. Thus, by the equality case of the monotonicity formula, $\K$ itself must be a round shrinking sphere.
\end{proof}

\begin{remark}
The proof illustrates the efficiency of the Andrews condition combined with Huisken's monotonicity formula. Alternatively (and more complicatedly), the lemma can also be proved by combining a couple of results from the literature:
 By a result of Hamilton \cite{Hamilton_convexpinched} in dimension $N>2$, respectively an ad hoc argument for $N=2$, the solution must be compact.
 Then, by results of Huisken-Sinestrari \cite{HS_OWR} and Daskalopoulos-Hamilton-Sesum \cite{DHS_curve_shorten} it must be a round.
\end{remark}

\begin{proof}[Proof of Theorem \ref{thm-intro_cylindrical} (smooth case)]
Suppose the cylindrical estimate doesn't hold. Then there exists a sequence $\K^j$ of $\alpha$-Andrews flows with $\lambda_1+\ldots + \lambda_k\geq \beta H$, such that
$(0,0)\in\partial K_0^j$, $H(0,0)=1$, $\K^j$ is defined in $P(0,0,j)$, and
$(\lambda_1+\ldots +\lambda_{k-1})(0,0)< j^{-1}$, but $\K^j$ is not $\varepsilon$-close to a cylinder.
By Theorem \ref{thm-intro_h_gives_global_convergence} we can pass to a smooth convex limit $\K$.
Since for this limit we have $(\lambda_1+\ldots +\lambda_{k-1})$(0,0)=0 and $\lambda_1+\ldots + \lambda_k\geq \beta H$, by the equality case of the maximum principle (see Appendix \ref{app_preserved_curv}) it follows that $\lambda_1+\ldots +\lambda_{k-1}$ vanishes identically on $\K$ and that we have an isometric splitting $K_t=\R^{k-1}\times N_t$ for some stricly convex flow $\mathcal{N}$ satisfying $\lambda_1/H\geq \beta$.
By Lemma \ref{lemma_sphere}, $\mathcal{N}$ must be a round shrinking sphere $S^{N-k}$. This implies that $\K^j$ is $\varepsilon$-close to a shrinking round cylinder $\mathbb{R}^{k-1}\times S^{N-k}$ for large $j$; a contradiction.
\end{proof}

\section{Approximation by smooth flows and consequences}\label{sec_structure_regularity}

\subsection{Background on elliptic regularization}
\label{subsec_elliptic_regularization_background}
Suppose $K_0\subset \R^N$ is smooth, compact and strictly mean convex and let $\{K_t\}_{t\geq 0}$ the level set flow starting at $K_0$, c.f. \cite{evans-spruck,CGG}.
The flow is nonfattening. It can be described by a Lipschitz function $u:K_0\to\R$ (the time of arrival function) such that $u(x)=t\Leftrightarrow x\in\D K_t$. The function $u$ satisfies
\begin{equation}
 \textrm{div} \left(\frac{Du}{\abs{Du}}\right)=-\frac{1}{\abs{D u}}
\end{equation}
in the viscosity sense.
 
By elliptic regularization $\{K_t\times \R \}$ arises as limit for $\eps\to 0$ of a family of smooth flows $\{ L_t^\eps \}$, c.f. \cite{evans-spruck,CGG,Ilmanen,Metzger_Schulze,white_subsequent}; we will describe this now.

Since $K_0$ is strictly mean convex, there are smooth solutions $u_\eps: K_0\to [0,\infty)$ of the Dirichlet problem
\begin{equation}\label{graphical_translator}
 \textrm{div}\left(\frac{D u_{\eps}}{\sqrt{\eps^2+\abs{Du_{\eps}}^2}}\right)=-\frac{1}{\sqrt{\eps^2+\abs{Du_{\eps}}^2}},\quad u_\eps|_{\D K_0}=0\ ,
\end{equation}
and as $\eps$ tends to zero the functions $u_\eps$ converge uniformly to $u$ \cite[Sec. 7.3]{evans-spruck}.

Geometrically, equation (\ref{graphical_translator}) says that $N_0^\eps=\textrm{graph}(\frac{1}{\eps}u_\eps)$ satisfies
\begin{equation}\label{eulerlagrange}
\vec{H}=-\frac{1}{\eps} e_{N+1}^\perp\, ,
\end{equation}
or equivalently that $N_t^\eps=\textrm{graph}\left(\frac{u_\eps-t}{\eps}\right)$ is a translating solution of the mean curvature flow.

Now for $(x,y)\in N_t^\eps$ we have $t=u_\eps(x)-\eps y$. Thus the time of arrival function of $\{N_t^\eps\}$ is given by
\begin{equation}
 U_\eps(x,y)=u_\eps(x)-\eps y \, .
\end{equation}
For $\eps\to 0$ it converges locally uniformly to $U(x,y)=u(x)$, which is the time of arrival function of $\{\partial K_t\times \R\}$.

To get rid of the boundary of the hypersurfaces $N_t^\eps$, we actually have to wait for a little time $\tau>0$. For $t\geq \tau$ let $L_t^\eps\subset B(0,\frac{\tau}{\eps})\subset\R^{N+1}$ be the region bounded by $N_t^\eps$, i.e. $\D L_t^\eps=N_t^\eps\cap B(0,\frac{\tau}{\eps})$.
Then for $\eps\to 0$ the flows $\{L_t^\eps\}_{t\geq \tau}$ converge to $\{K_t\times \R\}_{t\geq \tau}$.
The convergence is in the sense of Hausdorff convergence of $\L^\eps\subset\R^{N+1,1}$, and similarly for the complements.\footnote{It also convergences in the sense of Brakke flows, but we don't need that.} For $\tau$ small enough (depending on $K_0$), the convergence is smooth at least until $t=2\tau$.

\subsection{Properties of the elliptic approximators}

In this section, we prove that the elliptic approximators satisfy the Andrews condition. Recalling the geometric setting of Andrews \cite{andrews1},
let $K\subset \R^N$ be a smooth $\al$-Andrews domain and let $x\in\D K$. The interior ball $\bar{B}_{\Int}$ of radius $r(x)=\frac{\alpha}{H(x)}$ has center $c(x)=x-r(x)\nu_x$, where $\nu_x$ is the outward unit normal at $x$.
Thus, the interior noncollapsing condition $\bar{B}_{\Int}\subseteq K$ can be expressed as the inequality
\begin{equation}
\norm{y-c(x)}^2- r(x)^2=\norm{y-x}^2+2\tfrac{\alpha}{H(x)}\langle y-x,\nu_x\rangle\geq 0
\end{equation}
for all $y\in \D K$. Exterior noncollapsing works similarly.

\begin{theorem}\label{thm_ell_reg}
Let $K_0\subset \R^N$ be smooth compact $\al$-Andrews domain. Let $\L^\eps$ be the elliptic approximators as in the previous section.
 \begin{enumerate}
  \item The elliptic approximators $\L^\eps$ are $\al_\eps$-Andrews flows with constants $\al_\eps>0$ satisfying $\liminf_{\eps\to 0}\al_\eps\geq\alpha$.
  \item If $K_0$ satisfies in addition $\lambda_1+\ldots+\lambda_k\geq \beta H$, then $\L^\eps$ satisfies $\lambda_1+\ldots+\lambda_{k+1}\geq \beta_\eps H$ with $\liminf_{\beta\to 0}\beta_\eps\geq\beta$.
 \end{enumerate}

\end{theorem}

\begin{proof}
(1) We adopt the proof of Andrews-Langford-McCoy \cite{Andrews_Langford_McCoy} to our setting. Consider the funcions $Z^\ast,Z_\ast:N_0^\eps\to \R$ defined as
\begin{equation}
  Z^\ast(x):=\sup_{y\neq x} Z(x,y),\qquad Z_\ast(x):=\inf_{y\neq x} Z(x,y),
\end{equation}
where $Z(x,y)=2 \langle x-y,\nu_x\rangle \norm{x-y}^{-2}$.
Note that interior and exterior noncollapsing correspond to the inequalities $Z^\ast \leq \frac{1}{\al_\eps} H$ and $Z_\ast\geq -\frac{1}{\al_\eps} H$ respectively.
It follows from \cite[Thm. 2]{Andrews_Langford_McCoy} that
\begin{equation}
\Lap\frac{Z^\ast}{H}+2\langle \nabla\log H,\nabla \frac{Z^\ast}{H}\rangle \geq 0, \quad
\Lap\frac{Z_\ast}{H}+2\langle \nabla\log H,\nabla \frac{Z_\ast}{H}\rangle \leq 0,
\end{equation}
in the viscosity sense. Thus $\frac{Z^\ast}{H}$ attains its maximum over $N_0^\eps$ at the boundary $\partial N_0^\eps$, and likewise $\frac{Z_\ast}{H}$ attains its minimum over $N_0^\eps$ at the boundary $\partial N_0^\eps$.
Since $N_t^\eps$ moves by translation, the value of its optimal Andrews constant is independent of time. 
Thus, to find the limiting behavior of the Andrews constants for $\eps\to 0$, we can simply look at arbitrarily small times $t=\tau$. Since  
$\{L_t^\eps\}$ converges to $\{K_t\times \R\}$ as $\eps$ tends to zero, and since $K_0$ has Andrews constant $\alpha$, the claim follows.

\noindent(2) (cf. Cheeger-Haslhofer-Naber \cite[Sec. 4]{CHN}) On $N_0^\eps$ we have the equation
\begin{equation}
 \Lap\frac{A}{H}+2\langle \nabla \log H, \nabla \frac{A}{H}\rangle = 0.
\end{equation}
Note that $(\lambda_1+\ldots+\lambda_{k+1})\geq \beta_\eps H$ defines a convex subset of the space of symmetric matrices with positive trace.
Thus, the tensor maximum principle (see e.g. \cite[Sec. 4]{Hamilton_tensor_maximum}) implies that the minimum of the function $(\lambda_1+\ldots+\lambda_{k+1})/H$ over $N_0^\eps$
is attained on $\D N_0^\eps$. Recalling that the translators $\{L_t^\eps\}$ converge to $\{K_t\times \R\}$ as $\eps$ tends to zero, and observing that
splitting off the $\R$-factor corresponds to replacing $k+1$ by $k$, the claim follows.
\end{proof}

\subsection{Strong Hausdorff convergence and consequences}\label{subsec_strongHD}
In this section, we show that the results established previously in the smooth setting, hold for \emph{all} $\al$-Andrews flows, not just the smooth ones.
We also draw some further conclusions.

The key is the following notion of convergence.

\begin{definition}[Strong Hausdorff convergence]\label{def_strongHD}
 A sequence of $\al$-Andrews flows $\K^j$ converges to an $\al$-Andrews flow $\K$ in the \emph{strong Hausdorff sense}, if $\K^j$ Hausdorff converges to $\K$ and the complements of $\K^j$ Hausdorff converge to the closure of the complement of $\K$.
\end{definition}

In the elliptic regularization (Section \ref{subsec_elliptic_regularization_background}), since the time of arrival functions converge, it is clear that the convergence $\L^\eps\to \{K_t\times \R\}$ is in the strong Hausdorff sense.
More generally, we have:

\begin{proposition}[Strong Hausdorff convergence]\label{prop_strongHD}
 Every Hausdorff limit of $\al$-Andrews flows is a strong Hausdorff limit.
\end{proposition}

\begin{remark}
 For general sets, Hausdorff convergence of course does not imply strong Hausdorff convergence (as an illustrative example, consider the unit disc with a couple of holes of size $1/j$).
\end{remark}

\begin{proof}[Proof of Proposition \ref{prop_strongHD} (smooth case)] The idea is that the Andrews condition ensures that the boundaries $\D\K^j$ cannot behave too wildly. This is made precise by the following claim: 

\begin{claim}[tameness of the boundary]\label{claim_strongHD}
There is a $\rho=\rho(\alpha)>0$ such that if $\K$ is a $\al$-Andrews flow
defined in $P(p,t,r)$ and 
$(p,t)\in\D\K$ is a boundary point, then both 
$P(p,t,r)\cap \K$ and $P(p,t,r)\setminus \K$ contain
parabolic  balls of radius $\rho\,r$.
\end{claim}

\begin{proof}[Proof of Claim \ref{claim_strongHD} (smooth case)]
By translation and rescaling, we may take $(p,t,r)=(0,0,1)$.  By two-sided
avoidance,
there is a universal $t_1\in (-\frac12,0)$ such that the time slice $K_{t_1}$ has a boundary 
point in $B(0,\frac12)$.  By the speed limit lemma (Lemma \ref{lem-speed_limit}), there is a 
$(p_2,t_2)\in B(0,\frac12)\times [t_1,\frac12 t_1]$ with mean curvature at most
$t_1^{-1}$.  The $\al$-Andrews condition and two-sided avoidance imply the lemma.
\end{proof}

Now, to prove the proposition, let $\K^j$ be a sequence of smooth $\al$-Andrews flows that Hausdorff converges to a limit $\K$ (which by definition is also an $\al$-Andrews flow, though possibly nonsmooth).
Recall that Hausdorff convergence means that $\lim_{j\to\infty}d(X,\K^j)=0$ for every $X\in\K$ and $\liminf_{j\to\infty} d(Y,\K^j)>0$ for every $Y\in \K^c$.

The second condition implies $\lim_{j\to\infty}d(Y,(\K^j)^c)=0$ for every $Y\in\ol{\K^c}$.
We claim that $\liminf_{j\to\infty} d(X,(\K^j)^c)>0$ for every $X\in \Int(\K)$.
If not, after passing to a subsequence there would be boundary points $X_j\in \D\K^j$ converging to $X$, and by Claim \ref{claim_strongHD} this would
produce nearby parabolic balls in the complement of $\K^j$, contradicting the assumption that $\lim_{j\to\infty}d(X',\K^j)=0$ for every $X'\in\K$. Thus, $(\K^j)^c$ Hausdorff converges to $\ol{\K^c}$.
\end{proof}

\begin{corollary}\label{cor_strongHD}
 If a sequence of $\al$-Andrews flows Hausdorff converges, $\K^j\to\K$, then:
\begin{enumerate}
 \item The boundaries Hausdorff converge to the boundary of the limit, $\D\K^j\to\D\K$.
 \item Every compact subset of the interior of $\K$ is contained in the interior of $\K^j$ for $j$ sufficiently large.
\end{enumerate}
\end{corollary}

\begin{proof}[Proof (smooth case)]
 This immediately follows from Proposition \ref{prop_strongHD} and the definition of strong Hausdorff convergence, Definition \ref{def_strongHD}.
\end{proof}

\begin{remark}
 In Definition \ref{def_strongHD}, instead of requiring that the complements Hausdorff converge, equivalently one could require (1) or (2).
\end{remark}

The strong Hausdorff convergence ensures that all our results established previously in the smooth setting pass to Hausdorff limits:

\begin{claim}
 The results established previously for smooth $\al$-Andrews flows, hold for \emph{all} $\al$-Andrews flows.
\end{claim}

\begin{proof}
 By Definition \ref{def_andrews_flows}, Theorem \ref{thm_ell_reg} and Proposition \ref{prop_strongHD} every $\al$-Andrews flow, respectively its stabilized version $\{K_t\times \R\}$, is a strong Hausdorff limit of smooth flows satisfying the Andrews condition.
 The claim then follows from the following two basic observations:
\begin{itemize}
 \item If a result holds for a sequence of flows $\K^j$, and $\K^j$ converges to $\K$ in the strong Hausdorff sense, then it also holds for $\K$.
 \item If a result holds for $\{K_t\times \R\}$, then it also holds for $\{K_t\}$.
\end{itemize}
Indeed, if $\K$ is a Hausdorff limit of smooth $\al$-Andrews flows $\K^j$, then Corollary \ref{cor_strongHD} ensures that the curvature estimate (Theorem \ref{thm-intro_local_curvature_bounds}) holds for $\K$ as well; in particular points with finite viscosity mean curvature are smooth and the convergence is smooth in a neighborhood.
Once the curvature estimate is established, the rest is immediate.
\end{proof}

In particular, since we now know that Theorem \ref{thm-intro_smooth_convex_til_extinct} 
and Theorem  \ref{thm-intro_h_gives_global_convergence} hold for all $\al$-Andrews flows, we can apply them to blowup sequences to recover the main results from \cite{white_nature} (see also \cite{white_subsequent}):

\begin{corollary}[Limit flows]\label{cor_white_limits}
If $\K$ is a compact level set flow with smooth mean convex initial condition, then all its limit flows $\K'$ are smooth and convex until they disappear.
Furthermore, if $\K'$ is backwardly self-similar, then it is either (i) a static
halfspace or (ii) a shrinking round sphere or cylinder.
\end{corollary}

\begin{corollary}[Normalized limits]\label{cor_white_normalized_limits}
 Let $\K$ be a compact level set flow with smooth mean convex initial condition and let $(p_i,t_i)\in\D\K$ be a sequence of regular points with $H(p_i,t_i)\to\infty$.
Then the domains $\hat K_i$ obtained by translating $K_{t_i}$ by $-p_i$ and rescaling by $H(p_i,t_i)$ converge smoothly (modulo subsequence) to a smooth convex domain $\hat K$.
\end{corollary}

Moreover, we can now give the proof, also  without Brakke flows, of the partial regularity result, Theorem \ref{thm_intro_partial_regularity} (c.f. \cite{white_size}):

\begin{proof}[Proof of Theorem \ref{thm_intro_partial_regularity}]
We prove both the theorem and the refinement for $k$-convex flows simultaneously.
By the localized version of Huisken's monotonicity formula (see Appendix \ref{app_huisken_monotonicity}; note that we need (\ref{app_loc_mon}) only in the smooth setting, since we can always apply it for the smooth approximators) and Theorem \ref{thm-intro_smooth_convex_til_extinct} every tangent flow must be be either (i) a static
multiplicity one plane or (ii) a shrinking sphere or cylinder $\R^j\times B^{N-j}$ with $j<k$.
By Theorem \ref{thm-halfspace_smooth_convergence} the singular set $\S\subset \D\K$ consists exactly of those boundary points where no tangent flow is a static halfspace.
Assume the parabolic Hausdorff dimension of $\S$ is bigger than $k-1$. Then, invoking Theorem \ref{thm-halfspace_smooth_convergence} again and blowing up at a density point we obtain a tangent flow whose singular set has parabolic Hausdorff dimension bigger than $k-1$; this contradicts the above classification of tangent flows.
\end{proof}

Finally, the strong Hausdorff convergence also ensures that $\al$-Andrews flows inherit many desirable properties from smooth $\al$-Andrews flows, in particular we have:

\begin{corollary}\label{properties_of_alpha_andrews}
Every $\alpha$-Andrews flow $\{K_t\}_{t\in I}$ defined in an open set satisfies the following properties:
\begin{enumerate}
\item (Two-sided avoidance)  $\{K_t\}_{t\in I}$ avoids compact smooth mean curvature flows
from the inside and outside:  if $J\subset I$ is a subinterval 
and $\{K_t'\subset U\}_{t\in J}$ is a family of compact
smooth domains moving by mean curvature flow which is 
initially disjoint from or initially contained in $\{K_t\}$, then it remains disjoint
from or contained in $\{K_t\}$, respectively.
\item (Mean convexity)  The family $\{K_t\}$ is monotonic, in the sense that
$\; t_2\geq t_1\;\implies\; K_{t_2}\subseteq K_{t_1}$.
\item (Andrews condition) $K_t$ satisfies the viscosity $\al$-Andrews condition for every $t\in I$. 
\item ($\eps$-regularity) There are $\eps,\rho>0$ such that if both $\{K_t\}$ and
its complement $\{U\setminus K_t\}$ are  $(1+\eps)r$-
Hausdorff close to a halfspace (and its complement)
in a some parabolic ball $P(p,t,r)$,
then $\{K_t\}$ is smooth in the parabolic ball $P(p,t,\rho r)$. 
\end{enumerate}
\end{corollary}

\begin{remark}\label{axiom_remark}
In this paper we have chosen the approach of proving the results for smooth flows first and passing them to the limit afterwards.
Alternatively, and somewhat more axiomatically, we could have started with a family of closed sets $\{K_t\}_{t\in I}$ satisfying the properties (1)--(4). Our proofs carry over to this more general setting almost verbatim.
\end{remark}

\appendix

\section{Preserved curvature conditions and rigidity}\label{app_preserved_curv}

Let $\{M_t\}$ be a smooth mean curvature flow of oriented hypersurfaces. The evolution equation for the position vector, $\partial_t x =\vec{H}(x)$, implies evolution equations for all other geometric quantities (see e.g. \cite[Sec. 3]{Huisken_convex}), in particular
\begin{equation}\label{app_eqnH}
 \partial_t H=\Lap H +\abs{A}^2 H,
\end{equation}
for the mean curvature $H$ (here $\abs{A}^2$ denotes the square norm of $A$, in other words the sum of the squares of the principal curvatures), and
\begin{equation}\label{app_eqnA}
 \partial_t A=\Lap A +\abs{A}^2 A,
\end{equation}
for the second fundamental form (aka Weingarten map) $A=A_i^j$.

Equation (\ref{app_eqnH}) immediately implies, for let's say closed $\{M_t\}$, that the conditions $H\geq 0$ and $H>0$ are preserved along the flow, i.e. if they hold at the initial time $t=0$, then they hold for all $t\geq 0$; in fact, the reader can easily see from (\ref{app_eqnH}) that the minimum of the mean curvature is nondecreasing in time (provided it is nonnegative).

Similarly, appling the tensor maximum principle (see e.g. \cite[Sec. 4]{Hamilton_tensor_maximum}) to equation (\ref{app_eqnA}), one gets that the conditions
$\lambda_1+\ldots +\lambda_k \geq 0$, $\lambda_1+\ldots +\lambda_k > 0$ and $\lambda_1+\ldots +\lambda_k \geq \beta H$ are also preserved along the flow.
To see this, one merely has to observe that $\lambda_1+\ldots \lambda_k$ is a concave function on the space of matrices, c.f. \cite[Prop. 2.6]{huisken-sinestrari3}.

Finally, let us discuss the rigidity in the equality case of the maximum principle. This discussion is -- in contrast to the previous one -- purely local.
For (\ref{app_eqnH}) the strict maximum principle says that in a mean convex flow, $H$ must be stricly positive at every point, unless $H\equiv 0$ in a parabolic neighborhood. More interestingly, an investigation of the rigidity for (\ref{app_eqnA}) gives the following:

\begin{proposition}[Strict maximum principle for $\frac{\lambda_1}{H}$]
If $\{M_t\}$ is a smooth strictly mean convex mean curvature flow in $P(p,t,r)$ and $\frac{\lambda_1}{H}$ attains its minimum $\gamma$ over $P(p,t,r)$ at $(p,t)$,
then $\gamma$ must be nonnegative. Furthermore, if $\gamma=0$ then $\{M_t\}$ locally splits off a line in direction of any zero-eigenvector. In particular, the final time slice $M_t$ cannot be locally isometric to a part of a cone.
\end{proposition}

\begin{proof}(c.f. \cite[Sec. 8]{Hamilton_tensor_maximum}, \cite[App. A]{white_nature}).
Assume $\gamma<0$. Since $\lambda_1<0$ and $\lambda_{N-1}>0$ the Gauss curvature $K=\lambda_1\lambda_{N-1}$ is strictly negative.
However, by the strict maximum principle, the hypersurface locally splits as a product and thus this Gauss curvature must vanish; a contradiction.
Furthermore, if $\gamma=0$ then the strict maximum principle again enforces an isometric splitting.
In particular, moving along the radial direction of a cone we have that $\lambda_1=0$, but that $H$ scales like $r^{-1}$; since $H$ is nonzero, this contradicts the product structure. 
\end{proof}

\begin{corollary}
 If ${M_t\subset \R^N}$ is a smooth mean curvature flow of convex hypersurfaces such that $\lambda_1+\ldots \lambda_k\geq \beta H>0$ everywhere ($\beta>0$), and $\lambda_1+\ldots+\lambda_{k-1}=0$ at some point, then
$M_t=\R^{k-1}\times N_t$ for some stricly convex $N_t$.
\end{corollary}

\section{Huisken's monotonicity formula}\label{app_huisken_monotonicity}

Let $\{M_t\subset \R^N\}$ be a smooth mean curvature flow of hypersurfaces, say with at most polynomial volume growth.
Let $X_0=(x_0,t_0)\in \R^{N,1}$, and let
\begin{equation}
 \rho_{X_0}(x,t)=(4\pi(t_0-t))^{-(N-1)/2} e^{-\frac{\abs{x-x_0}^2}{4(t_0-t)}}\qquad (t<t_0),
\end{equation}
be the $(N-1)$-dimensional backwards heat kernel centered at $X_0$. Notice that $\rho_{X_0}$ is normalized such that the integral over any $(N-1)$-plane containing $X_0$ (e.g. over $T_{x_0}M_{t_0}$ in case $X_0\in \M$) equals one.
Huisken discovered the monotonicity formula \cite[Thm 3.1]{Huisken_monotonicity},
\begin{equation}
 \frac{d}{dt}\int_{M_t} \rho_{X_0} dA = -\int_{M_t} \left|\vec{H}-\frac{(x-x_0)^\perp}{2(t-t_0)}\right|^2 \rho_{X_0} dA\qquad (t<t_0).
\end{equation}
Huisken's monotonicity formula can be thought of as localized version of the fact that the mean curvature flow is the gradient flow of the area functional.
Also, the equality case exactly characterized the selfsimilarly shrinking solutions of the mean curvature flow, i.e. assuming $X_0=(0,0)$ the solutions satisfying $M_t=\sqrt{-t}M_{-1}$.

If $M_t$ is only defined locally, say in $B(x_0,\sqrt{2N}\rho)\times (t_0-\rho^2,t_0)$, then one can use the cutoff function
\begin{equation}
 \varphi_{X_0}^\rho(x,t)=\left(1-\frac{\abs{x-x_0}^2+2(N-1)(t-t_0)}{\rho^2}\right)_+^3
\end{equation}
and still gets the monotonicity inequality \cite[Prop. 4.17]{Ecker_book},
\begin{equation}\label{app_loc_mon}
 \frac{d}{dt}\int_{M_t} \rho_{X_0}\varphi_{X_0}^\rho dA \leq -\int_{M_t} \left|\vec{H}-\frac{(x-x_0)^\perp}{2(t-t_0)}\right|^2 \rho_{X_0}\varphi_{X_0}^\rho dA.
\end{equation}
The monotone quantity appearing on the left hand side,
\begin{equation}
\Theta^\rho(\M,X_0,r)=\int_{M_{t_0-r^2}} \rho_{X_0}\varphi_{X_0}^\rho dA,
\end{equation}
is called the Gaussian density ratio, and plays an important role in the local regularity theorem.

\section{The local regularity theorem}\label{app_easy_brakke}

\begin{theorem}[Easy Brakke, \cite{white_regularity}]\label{app_thm_easy_brakke}
There exist universal constants $\eps>0$ and $C<\infty$ with the following property.
If $\M$ is a smooth mean curvature flow of hypersurfaces in a parabolic ball $P(X_0,2N\rho)$ with
\begin{equation}
 \sup_{X\in P(X_0,r)}\Theta^{\rho}(\M,X,r)<1+\eps
\end{equation}
for some $r\in(0,\rho)$, then
\begin{equation}
 \sup_{P(X_0,r/2)}\abs{A}\leq \frac{C}{r}.
\end{equation}
\end{theorem}

\begin{remark}
Of course, the theorem can also be applied to limits of smooth flows. Moreover, for applications it is useful to observe the following: If $\Theta<1+\frac{\eps}{2}$ holds at some point and some scale, then $\Theta<1+\eps$ holds at all nearby points and all slightly smaller scales.
\end{remark}

\begin{remark}
For general Brakke flows the proof of the local regularity theorem is very difficult \cite{brakke,kasai_tonegawa}. However, as pointed out by White \cite{white_regularity} (see also \cite{anderson,Ecker_book}), there is a fairly simple argument in the smooth setting.
For convenience of the reader, we will now give (an even somewhat more streamlined variant of) this argument.
\end{remark}

\begin{proof}[Proof of Theorem \ref{app_thm_easy_brakke}]
 Suppose the assertion fails. Then there exist a sequence of smooth flows $\M^j$ in $P(0,2N \rho_j)$ for some $\rho_j> 1$ such that
\begin{equation}
 \sup_{X\in P(0,1)}\Theta^{\rho_j}(\M^j,X,1)<1+\frac{1}{j},
\end{equation}
but such that there are points $X_j\in P(0,1/2)$ with $\abs{A}(X_j)> j$.

By point selection, we can find $Y_j\in P(0,3/4)$ with $Q_j=\abs{A}(Y_j)> j$ such that
\begin{equation}\label{app_brakke_point_sel}
 \sup_{P(Y_j,j/10Q_j)}\abs{A}\leq 2 Q_j.
\end{equation}
For convenience of the reader, let us explain how the point selection works: Fix $j$. If $Y^0_j=X_j$ already satisfies (\ref{app_brakke_point_sel}) with $Q^0_j=\abs{A}(Y^0_j)$, we are done. Otherwise, there is a point $Y^1_j\in P(Y_j^0,j/10Q^0_j)$ with $Q^1_j=\abs{A}(Y^1_j)>2Q^0_j$.
If $Y^1_j$ satisfies (\ref{app_brakke_point_sel}), we are done. Otherwise, there is a point $Y^2_j\in P(Y_j^1,j/10Q^1_j)$ with $Q^2_j=\abs{A}(Y^2_j)>2Q^1_j$, etc.
Note that $\frac{1}{2}+\frac{j}{10Q_j^0}(1+\frac{1}{2}+\frac{1}{4}+\ldots)<\frac{3}{4}$. By smoothness, the iteration terminates after a finite number of steps, and the last point of the iteration lies in $P(0,3/4)$ and satisfies (\ref{app_brakke_point_sel}).

Continuing the proof of the theorem, let $\hat\M^j$ be the flows obtained by shifting $Y_j$ to the origin and parabolically rescaling by $Q_j=\abs{A}(Y_j)\to\infty$.
Since $\abs{A}(0)=1$ and $\sup_{P(0,j/10)}\abs{A}\leq 2$, we can pass smoothly to a nonflat global limit. On the other hand, by the rigidity case of (\ref{app_loc_mon}), and since
\begin{equation}
 \Theta^{\hat\rho_j}(\hat\M^j,0,Q_j)<1+j^{-1},
\end{equation}
where $\hat\rho_j=Q_j\rho_j\to\infty$, the limit is a flat plane; a contradiction.
\end{proof}

\section{Removing the admissibility assumption}\label{app_remove_admissibility}

In Section \ref{subsec_halfspace_conv} and Section \ref{subsec_curv_est} we worked under the technical assumption that some time slice $K_{\bar t}$ contains $B(p,r)$. We will now prove that this assumption actually can be removed.

Let $\rho=\rho(\al)$, $C_0=C_0(\al)$ be the constants from the admissible version of the curvature estimate, Theorem 1.8'.

\begin{claim}\label{claim_adm}
There is an $\eta=\eta(\al)<\infty$ such that if $\K$ is a smooth $\al$-Andrews flow
defined in $P(p,t,\eta r)$ with $H(p,t)\leq r^{-1}$, then $\abs{A}\leq 2C_0r^{-1}$
in $P(p,t,\frac{\rho}{2} r)$.
\end{claim}

\begin{proof}
If not, there would be a sequence $\{\K^j\}$
of counterexamples with $\eta_j=j\ra \infty$.  By point picking, parabolic rescaling, and passing to a subsequence we may assume that:
\begin{enumerate}
\item $\K^j$ is defined in $P(0,0,j)$.
\item $H(0,0)\ra H_\infty\leq 1$ as $j\ra\infty$.
\item $\sup \abs{A}>2C_0$ in $P(0,0,\frac{\rho}{2})$, for every $j$.
\item If the hypotheses of the claim
hold for some $(p,t)\in P(0,0,j/5)$ and $r\leq \frac12$, then the conclusion holds as well.
\end{enumerate}

For convenience of the reader, let us explain the point picking: Fix $j$. Let $(p_0,t_0)=(0,0)$. If (4) holds we are done. Otherwise we can find $(p_1,t_1)\in P(0,0,j/5)$ and $r_1\leq 1/2$ such that $H(p_1,t_1)\leq r_1^{-1}$ but $\abs{A}>2C_0r_1^{-1}$ somewhere in $P(p_1,t_1,\frac{\rho}{2}r_1)$. Take $(p_1,t_1)$ as new center and rescale to normalize $r_1$. Iterate this process. By smoothness, it terminates after a finite number of steps, i.e. eventually (4) is satisfied.
(3) is satisfied by construction, and since $\frac{j}{5} (1+\frac{1}{2}+\frac{1}{4}+\ldots)<\frac{j}{2}$ we see that (1) holds as well. After passing to a subsequence, we get (2).

Having explained the point picking, let us now continue the proof of the claim. Suppose $H_\infty=0$. Then by the first part of the proof of the halfspace convergence theorem, we get that $\K^j$ subconverges weakly
to a static halfspace.  
Applying (4) with center $(0,0)$ we get that the convergence is smooth on $P(0,0,\frac{\rho}{4})$. Applying (4) repeatedly with different centers we can extend the smooth convergence to larger and larger parabolic balls; this contradicts (3).  Thus $H_\infty>0$. 

By (4), we have $\abs{A}\leq 4C_0$ in $P(0,0,\frac{\rho}{4})$.  Therefore, by standard gradient
estimates (see e.g. \cite[Prop. 3.22]{Ecker_book}) we get $H\geq \frac{H_\infty}{2}$ in $P(0,0,\rho_1)$, for $\rho_1=\rho_1(\rho,C_0,H_\infty)$.
It follows that there exist $r_0>0$ and $t_0<0$ independent of $j$ such
that  $B(0,r_0)\subset K_{t_0}$ for large $j$. But then by Lemma \ref{lemma_admissibility} below, there exists $\bar t>-\infty$ such that $B(0,1)\subset K_{\bar t}$ for
large $j$, i.e. $P(0,1)$ is admissible for large $j$. Thus, Theorem 1.8' implies $\abs{A}\leq 2C_0$ in 
$P(0,0,\frac{\rho}{2})$ for large $j$; this contradicts (4).
\end{proof}

\begin{lemma}[Definite progress going backwards in time]\label{lemma_admissibility}
For every $\al>0$, $t_0<0$, $r_0>0$ and $r_1<\infty$ there exist $R=R(\al,t_0,r_0,r_1)<\infty$ and $t_1=t_1(\al,t_0,r_0,r_1)>-\infty$ such that if
$\K$ is an $\al$-Andrews flow in $P(0,0,R)$ with $0\in\D K_0$ and $B(0,r_0)\subset K_{t_0}$, then $B(0,r_1)\subset K_{t_1}$.
\end{lemma}

\begin{proof}[Proof (smooth case).]
The distance function
$f(t)=d(0,\D K_t)$ satisfies $\abs{f'(t)}\gtrsim \,\frac{\al r_0}{\abs{t}}$ for $t\leq t_0$.
Otherwise, the Andrews condition and comparison with spheres (which is legitimate provided $R$ is large enough) would give $f(0)>0$; a contradiction.
Thus, the claim follows by integration.
\end{proof}

By Claim \ref{claim_adm}, after adjusting $C_\ell$ and $\rho$, we get the curvature estimate without admissibility assumption, 
Theorem \ref{thm-intro_local_curvature_bounds}. 
As a consequence, in the halfspace convergence theorem, Theorem \ref{thm-halfspace_smooth_convergence}, it really suffices to impose the assumption (1) and there is no need for assumption (1').
Indeed, by interior contact with spheres one can find boundary points with arbitrarily small curvature, and then one can simply apply Theorem \ref{thm-intro_local_curvature_bounds}.

\bibliography{mean_convex_flow}

\newcommand{\noopsort}[1]{} \newcommand{\singleletter}[1]{#1}
\begin{thebibliography}{CGG91}

\bibitem[ALM13]{Andrews_Langford_McCoy}
B.~Andrews, M.~Langford, and J.~McCoy.
\newblock Non-collapsing in fully non-linear curvature flows.
\newblock {\em Ann. Inst. H. Poincar\'e Anal. Non Lin\'eaire}, 30(1):23--32,
  2013.

\bibitem[And90]{anderson}
M.~Anderson.
\newblock Convergence and rigidity of manifolds under {R}icci curvature bounds.
\newblock {\em Invent. Math.}, 102(2):429--445, 1990.

\bibitem[And12]{andrews1}
B.~Andrews.
\newblock Noncollapsing in mean-convex mean curvature flow.
\newblock {\em Geom. Topol.}, 16(3):1413--1418, 2012.

\bibitem[Bra78]{brakke}
K.~Brakke.
\newblock {\em The motion of a surface by its mean curvature}, volume~20 of
  {\em Mathematical Notes}.
\newblock Princeton University Press, Princeton, N.J., 1978.

\bibitem[CGG91]{CGG}
Y.G. Chen, Y.~Giga, and S.~Goto.
\newblock Uniqueness and existence of viscosity solutions of generalized mean
  curvature flow equations.
\newblock {\em J. Differential Geom.}, 33(3):749--786, 1991.

\bibitem[Che09]{chen_uniqueness}
B.-L. Chen.
\newblock Strong uniqueness of the {R}icci flow.
\newblock {\em J. Differential Geom.}, 82(2):363--382, 2009.

\bibitem[CHN13]{CHN}
J.~Cheeger, R.~Haslhofer, and A.~Naber.
\newblock Quantitative stratification and the regularity of mean curvature
  flow.
\newblock {\em Geom. Funct. Anal.}, 23(3):828--847, 2013.

\bibitem[CM12]{CM_generic}
T.~Colding and W.~Minicozzi.
\newblock Generic mean curvature flow {I}; generic singularities.
\newblock {\em Ann. of Math. (2)}, 175(2):755--833, 2012.

\bibitem[DHS10]{DHS_curve_shorten}
P.~Daskalopoulos, R.~Hamilton, and N.~Sesum.
\newblock Classification of compact ancient solutions to the curve shortening
  flow.
\newblock {\em J. Differential Geom.}, 84(3):455--464, 2010.

\bibitem[Eck]{Ecker_kconvex}
K.~Ecker.
\newblock Partial regularity at the first singular time for hypersurfaces
  evolving by mean curvature.
\newblock {\em Math. Ann. (to appear)}.

\bibitem[Eck04]{Ecker_book}
K.~Ecker.
\newblock {\em Regularity theory for mean curvature flow}.
\newblock Progress in Nonlinear Differential Equations and their Applications,
  57. Birkh\"auser Boston Inc., Boston, MA, 2004.

\bibitem[ES91]{evans-spruck}
L.~Evans and J.~Spruck.
\newblock Motion of level sets by mean curvature. {I}.
\newblock {\em J. Differential Geom.}, 33(3):635--681, 1991.

\bibitem[Ham86]{Hamilton_tensor_maximum}
R.~Hamilton.
\newblock Four-manifolds with positive curvature operator.
\newblock {\em J. Differential Geom.}, 24(2):153--179, 1986.

\bibitem[Ham94]{Hamilton_convexpinched}
R.~Hamilton.
\newblock Convex hypersurfaces with pinched second fundamental form.
\newblock {\em Comm. Anal. Geom.}, 2(1):167--172, 1994.

\bibitem[Ham95]{Hamilton_survey}
R.~Hamilton.
\newblock The formation of singularities in the {R}icci flow.
\newblock In {\em Surveys in differential geometry, {V}ol.\ {II} ({C}ambridge,
  {MA}, 1993)}, pages 7--136. Int. Press, Cambridge, MA, 1995.

\bibitem[Hea11]{Head_thesis}
J.~Head.
\newblock The surgery and level-set approaches to mean curvature flow.
\newblock {\em PhD-thesis, FU Berlin and AEI Potsdam}, 2011.

\bibitem[HK]{haslhofer-kleiner_surgery}
R.~Haslhofer and B.~Kleiner.
\newblock A new construction of mean curvature flow with surgery.
\newblock in preparation.

\bibitem[HS]{HS_OWR}
G.~Huisken and C.~Sinestrari.
\newblock Ancient solutions to mean curvature flow.
\newblock {\em Oberwolfach Report (to appear)}.

\bibitem[HS99a]{huisken-sinestrari1}
G.~Huisken and C.~Sinestrari.
\newblock Mean curvature flow singularities for mean convex surfaces.
\newblock {\em Calc. Var. Partial Differential Equations}, 8(1):1--14,
  {\noopsort{a}}1999.

\bibitem[HS99b]{huisken-sinestrari2}
G.~Huisken and C.~Sinestrari.
\newblock Convexity estimates for mean curvature flow and singularities of mean
  convex surfaces.
\newblock {\em Acta Math.}, 183(1):45--70, {\noopsort{b}}1999.

\bibitem[HS09]{huisken-sinestrari3}
G.~Huisken and C.~Sinestrari.
\newblock Mean curvature flow with surgeries of two-convex hypersurfaces.
\newblock {\em Invent. Math.}, 175(1):137--221, 2009.

\bibitem[Hui]{Huisken_privatecommunication}
G.~Huisken.
\newblock private communication.

\bibitem[Hui84]{Huisken_convex}
G.~Huisken.
\newblock Flow by mean curvature of convex surfaces into spheres.
\newblock {\em J. Differential Geom.}, 20(1):237--266, 1984.

\bibitem[Hui90]{Huisken_monotonicity}
G.~Huisken.
\newblock Asymptotic behavior for singularities of the mean curvature flow.
\newblock {\em J. Differential Geom.}, 31(1):285--299, 1990.

\bibitem[Hui93]{Huisken_local_global}
G.~Huisken.
\newblock Local and global behaviour of hypersurfaces moving by mean curvature.
\newblock In {\em Differential geometry: partial differential equations on
  manifolds ({L}os {A}ngeles, {CA}, 1990)}, volume~54 of {\em Proc. Sympos.
  Pure Math.}, pages 175--191. Amer. Math. Soc., Providence, RI, 1993.

\bibitem[Ilm94]{Ilmanen}
T.~Ilmanen.
\newblock Elliptic regularization and partial regularity for motion by mean
  curvature.
\newblock {\em Mem. Amer. Math. Soc.}, 108(520):x+90, 1994.

\bibitem[Kle92]{kleiner_isoperimetric_comparison}
B.~Kleiner.
\newblock An isoperimetric comparison theorem.
\newblock {\em Invent. Math.}, 108(1):37--47, 1992.

\bibitem[KT11]{kasai_tonegawa}
K.~Kasai and Y.~Tonegawa.
\newblock A general regularity theory for weak mean curvature flow.
\newblock arXiv:1111.0824, 2011.

\bibitem[Man11]{Mantegazza_book}
C.~Mantegazza.
\newblock {\em Lecture notes on mean curvature flow}, volume 290 of {\em
  Progress in Mathematics}.
\newblock Birkh\"auser/Springer Basel AG, Basel, 2011.

\bibitem[MS08]{Metzger_Schulze}
J.~Metzger and F.~Schulze.
\newblock No mass drop for mean curvature flow of mean convex hypersurfaces.
\newblock {\em Duke Math. J.}, 142(2):283--312, 2008.

\bibitem[Per02]{perelman_entropy}
G.~Perelman.
\newblock The entropy formula for the {R}icci flow and its geometric
  applications.
\newblock arXiv:math/0211159, 2002.

\bibitem[Per03]{perelman_surgery}
G.~Perelman.
\newblock Ricci flow with surgery on three-manifolds.
\newblock arXiv:math/0303109, 2003.

\bibitem[SW09]{sheng_wang}
W.~Sheng and X.~Wang.
\newblock Singularity profile in the mean curvature flow.
\newblock {\em Methods Appl. Anal.}, 16(2):139--155, 2009.

\bibitem[Wan11]{wang_convex}
X.~Wang.
\newblock Convex solutions to the mean curvature flow.
\newblock {\em Ann. of Math. (2)}, 173(3):1185--1239, 2011.

\bibitem[Whi00]{white_size}
B.~White.
\newblock The size of the singular set in mean curvature flow of mean convex
  sets.
\newblock {\em J. Amer. Math. Soc.}, 13(3):665--695, 2000.

\bibitem[Whi02]{White_ICM}
B.~White.
\newblock Evolution of curves and surfaces by mean curvature.
\newblock In {\em Proceedings of the {I}nternational {C}ongress of
  {M}athematicians, {V}ol. {I} ({B}eijing, 2002)}, pages 525--538, Beijing,
  2002. Higher Ed. Press.

\bibitem[Whi03]{white_nature}
B.~White.
\newblock The nature of singularities in mean curvature flow of mean-convex
  sets.
\newblock {\em J. Amer. Math. Soc.}, 16(1):123--138, 2003.

\bibitem[Whi05]{white_regularity}
B.~White.
\newblock A local regularity theorem for mean curvature flow.
\newblock {\em Ann. of Math. (2)}, 161(3):1487--1519, 2005.

\bibitem[Whi11]{white_subsequent}
B.~White.
\newblock Subsequent singularities in mean-convex mean curvature flow.
\newblock arXiv:1103.1469, 2011.

\end{thebibliography}

\bibliographystyle{alpha}

\vspace{10mm}
{\sc Courant Institute of Mathematical Sciences, New York University, 251 Mercer Street, New York, NY 10012, USA}

\emph{E-mail:} robert.haslhofer@cims.nyu.edu, bkleiner@cims.nyu.edu

\end{document}